\documentclass[]{article}

\usepackage{graphicx}
\usepackage{amsmath}
\usepackage{amssymb}
\usepackage{amsthm}
\usepackage{pxfonts}
\usepackage{enumerate}
\usepackage{color}
\usepackage{mathdots}
\usepackage{sectsty}
\usepackage[hidelinks]{hyperref}
\usepackage{tikz}
\usepackage{caption}
\usepackage{adjustbox}

\sectionfont{\scshape\centering\fontsize{11}{14}\selectfont}
\subsectionfont{\scshape\fontsize{11}{14}\selectfont}
\usepackage{fancyhdr}

\newcommand\shorttitle{Determinantal structures in space inhomogeneous dynamics on interlacing arrays}
\newcommand\authors{Theodoros Assiotis}

\newtheorem{thm}{Theorem}[section]

\newtheorem{lem}[thm]{Lemma}
\newtheorem{defn}[thm]{Definition}
\newtheorem{rmk}[thm]{Remark}
\newtheorem{prop}[thm]{Proposition}

\title{\large \bf DETERMINANTAL STRUCTURES IN SPACE INHOMOGENEOUS DYNAMICS ON INTERLACING ARRAYS}
\author{\small THEODOROS ASSIOTIS}
\date{}

\begin{document}

\maketitle

\begin{abstract}
We introduce a space inhomogeneous generalization of the dynamics on interlacing arrays considered by Borodin and Ferrari in \cite{BorodinFerrari}. We show that for a certain class of initial conditions the point process associated to the dynamics has determinantal correlation functions and we calculate explicitly, in the form of a double contour integral, the correlation kernel for one of the most classical initial conditions, the densely packed. En route to proving this we obtain some results of independent interest on non-intersecting general pure-birth chains, that generalize the Charlier process, the discrete analogue of Dyson's Brownian motion. Finally, these dynamics provide a coupling between the inhomogeneous versions of the TAZRP and PushTASEP particle systems which appear as projections on the left and right edges of the array respectively.
\end{abstract}

\tableofcontents

\section{Introduction}

\subsection{Informal introduction}
The study of stochastic dynamics, in both discrete and continuous time, on interlacing arrays has seen an enormous amount of activity in the past two decades, see for example \cite{BorodinFerrari}, \cite{BorodinKuan}, \cite{BorodinOlshanski}, \cite{BorodinPetrov}, \cite{MacdonaldProcesses}, \cite{WarrenWindridge}, \cite{SchurDynamics}, \cite{StochasticVertex}, \cite{Warren} \cite{InterlacingDiffusions}, \cite{RandomGrowthKarlinMcGregor}, \cite{BorodinGorin}, \cite{BorodinGorinRains}, \cite{BorodinFerrariTilings}. These dynamics can equivalently be viewed as growth of random surfaces, see \cite{BorodinFerrari}, \cite{BorodinKuan}, \cite{RandomGrowthKarlinMcGregor} or as random fields of Young diagrams, see \cite{BufetovPetrov}, \cite{BufetovPetrovMucciconi}. Currently there are arguably three main approaches in constructing dynamics on interlacing arrays with some underlying integrability\footnote{By this rather vague term we mean that there exist useful explicit formulae for the expectations of at least some observables. See also the introductions of \cite{BufetovPetrov} and \cite{BufetovPetrovMucciconi} for a lengthier historical overview and comparison between the three approaches.}. 

The one that we will be concerned with in this contribution is due to Borodin and Ferrari\footnote{Thus the name "Borodin-Ferrari push-block dynamics" that we use throughout the paper.} \cite{BorodinFerrari} (see also the independent related work of Warren and Windridge \cite{WarrenWindridge} and Warren's Brownian analogue of the dynamics \cite{Warren}) based on some ideas from \cite{DiaconisFill}. This is the simpler out of the three approaches to describe (see Definition \ref{BFDynamics} for a precise description) and in some sense, see \cite{BufetovPetrov}, the one with "maximal noise". Many of the ideas and results from the important paper \cite{BorodinFerrari} have either directly generated or have been made use of in a very large body of work, see for example \cite{BorodinKuan}, \cite{MacdonaldProcesses}, \cite{StochasticVertex}, \cite{SchurDynamics}, \cite{BorodinOlshanski}, \cite{Duits}, \cite{BorodinPetrov}, \cite{qWhittakerProcesses}, \cite{Toninelli}, \cite{ChhitaFerrari}, \cite{BorodinGorin}, \cite{BorodinGorinRains}, \cite{BorodinFerrariTilings} for further developments and closely related problems. The other two approaches can be concisely described as follows: one of them (which is historically the first out of the three) is based on the combinatorial algorithm of the RSK correspondence, see \cite{JohanssonShapeFluctuations}, \cite{OConnellTams}, \cite{OConnellJPhys}, \cite{Toda} and the other, which has begun to develop very recently, is based on the Yang-Baxter equation, see \cite{BufetovPetrov}, \cite{BufetovPetrovMucciconi}.

Now, in the past few years, there has been considerable interest in constructing new integrable models in inhomogeneous space or adding spatial inhomogeneities, in a natural way, to existing models while preserving the integrability, see \cite{BorodinPetrovInhomogeneous}, \cite{KnizelPetrovSaenz}, \cite{Emrah}, \cite{EmrahThesis}, \cite{RandomGrowthKarlinMcGregor}. In this paper we do exactly that for the original (continuous time) dynamics of Borodin and Ferrari (see Definition \ref{BFDynamics}). We show that for a certain class of initial conditions the point process associated to the dynamics has determinantal correlation functions. We then calculate explicitly, in the form of a double contour integral, the correlation kernel for one of the most classical initial conditions, the densely packed. This allows one to address questions regarding asymptotics and it would be interesting to return to this in future work. Here, our focus is on developing the stochastic integrability aspects of the model.

Finally, the projections on the edges of the interlacing array give two Markovian interacting particle systems of independent interest, see Remark \ref{EdgeSystemsRemark} for more details. On the left edge we get the inhomogeneous TAZRP (totally asymmetric zero range process) or Boson particle system, see \cite{MacdonaldProcesses}, \cite{StochasticVertex}, \cite{qBoson} and on the right edge we get the inhomogeneous PushTASEP, see \cite{BorodinFerrari}, \cite{BorodinFerrariPushASEP}, \cite{BorodinPetrov}, which is also studied in detail in the independent work of Leonid Petrov \cite{PetrovInhomogeneousPushTASEP} which uses different methods\footnote{Petrov finds an integrable structure underlying PushTASEP by making a connection to the theory of Schur processes, see \cite{SchurMeasures}, \cite{SchurProcess}. He also performs some asymptotics. The overlap in terms of results and techniques between the two papers is minuscule.}.

In the next subsection we give the necessary background in order to introduce the model and state our main results precisely.
\subsection{Background and main result} 
We define the discrete Weyl chamber with non-negative coordinates:
\begin{align*}
\mathbb{W}^N=\{(x_1,\dots,x_N)\in \mathbb{Z}_+^N:x_1<\dots<x_N\},
\end{align*}
where $\mathbb{Z}_+=\{0,1,2,\dots\ \}$.

We think of the coordinates $x_i$ as positions of particles and will use this terminology throughout, see Figure \ref{GelfandTsetlinDynamics} for an illustration. We say that $y \in \mathbb{W}^N$ and $x \in \mathbb{W}^{N+1}$ interlace and write $y\prec x$ if:
\begin{align*}
 x_1 \le y_1 < x_2 \le \dots  <x_{N}\le y_N< x_{N+1}.
\end{align*}
We define the set of Gelfand-Tsetlin patterns (interlacing arrays) of length $N$ by\footnote{The definition of interlacing, and thus of Gelfand-Tsetlin patterns, is slightly different to the one used in \cite{BorodinFerrari}; more precisely the positions of inequalities $\le$ and strict inequalities $<$ are swapped. Clearly the two conventions are equivalent (after a simple shift).}:
\begin{align}
\mathsf{GT}_{N}=\big\{\left(x^{(1)},\dots,x^{(N)}\right): x^{(i)} \in \mathbb{W}^{i}, \ x^{(i)}\prec x^{(i+1)},\textnormal{ for } 1 \le i\le N-1 \big\}.
\end{align}
The basic data in this paper is a rate function 
\begin{align*}
\lambda:\mathbb{Z}_+\to (0,\infty)
\end{align*}
which we think of as the spatial inhomogeneity of the environment. It governs how fast or slow particles jump when at a certain position. We enforce the following assumption throughout the paper.
\begin{defn}\textsf{Assumption (UB)}. We assume that the rate function $\lambda :\mathbb{Z}_+\to (0,\infty)$ is uniformly bounded away from $0$ and $\infty$:
\begin{align}
0<s\overset{\textnormal{def}}{=}\inf_{x\ge 0}\lambda(x)\le \sup_{x\ge 0}\lambda(x)\overset{\textnormal{def}}{=}M <\infty.
\end{align}
\end{defn}

We now introduce the inhomogeneous space push-block dynamics in $\mathsf{GT}_N$. This is the continuous time Markov jump process in $\mathsf{GT}_N$ described as follows:

\begin{defn}\label{BFDynamics}\textsf{Borodin-Ferrari inhomogeneous space push-block dynamics}. Let $\lambda(\cdot)$ satisfy \textsf{(UB)}. Let $\mathsf{M}_N$ be the initial distribution (possibly deterministic) of particles on $\mathsf{GT}_{N}$. We now describe Markov dynamics in $\mathsf{GT}_N$ denoted by $\left(\mathsf{X}_N(t;\mathsf{M}_N);t\ge0\right)=\left(\left(\mathsf{X}^1(t),\dots,\mathsf{X}^N(t)\right);t\ge 0\right)$ where the projection on the $k^{th}$ level is given by $\left(\mathsf{X}^k(t);t\ge 0\right)=\left(\left(\mathsf{X}^{k}_1(t),\mathsf{X}^{k}_2(t),\dots,\mathsf{X}^{k}_k(t)\right);t\ge 0\right)$.

 Each particle has an independent exponential clock of rate $\lambda(\star)$ depending on its current position $\star \in \mathbb{Z}_+$ for jumping to the right by one to site $\star+1$. The particles interact as follows, see Figure \ref{GelfandTsetlinDynamics} for an illustration: if the clock of particle $\mathsf{X}_k^{n}$ rings first then it will attempt to jump to the right by one.
\begin{itemize}
\item (\textsf{Blocking}) In case $\mathsf{X}_k^{n-1}=\mathsf{X}_k^{n}$ the jump is blocked (since a move would break the interlacing; lower level particles can be thought of as heavier).
\item  (\textsf{Pushing}) Otherwise it moves by one to the right, possibly triggering instantaneously some pushing moves to maintain the interlacing. Namely, if the interlacing is no longer preserved with the particle labelled $\mathsf{X}_{k+1}^{n+1}$ then, $\mathsf{X}_{k+1}^{n+1}$ also moves instantaneously to the right by one and this pushing is propagated (instantaneously) to higher levels, if needed.
\end{itemize}

\end{defn}

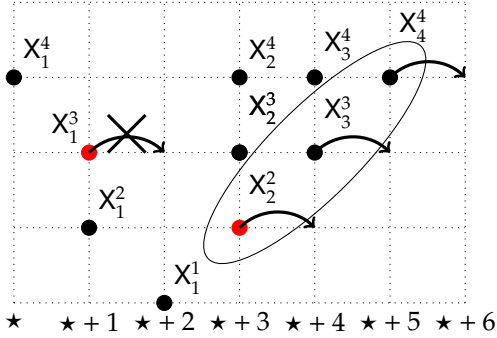
\begin{figure}
\captionsetup{singlelinecheck = false, justification=justified}
\begin{tikzpicture}
\draw[dotted] (0,0) grid (6,4);

\draw[fill] (2,0) circle [radius=0.1];
\node[above right] at (2,0) {$\mathsf{X}_1^{1}$};

\draw[fill] (1,1) circle [radius=0.1];
\node[above right] at (1,1) {$\mathsf{X}_1^{2}$};

\draw[fill,red] (3,1) circle [radius=0.1];
\node[above right] at (3,1.2) {$\mathsf{X}_2^{2}$};

\draw[fill] (4,2) circle [radius=0.1];
\node[above right] at (4,2.2) {$\mathsf{X}_3^{3}$};

\draw[fill] (5,3) circle [radius=0.1];
\node[above right] at (5,3.35) {$\mathsf{X}_4^{4}$};

\draw[fill] (3,2) circle [radius=0.1];
\node[above right] at (3,2.25) {$\mathsf{X}_2^{3}$};

\draw[fill] (3,2) circle [radius=0.1];
\node[above right] at (3,2.25) {$\mathsf{X}_2^{3}$};

\draw[fill] (4,3) circle [radius=0.1];
\node[above right] at (4,3.1) {$\mathsf{X}_3^{4}$};

\draw[fill] (3,3) circle [radius=0.1];
\node[above right] at (3,3) {$\mathsf{X}_2^{4}$};

\draw[fill] (0,3) circle [radius=0.1];
\node[above right] at (0,3) {$\mathsf{X}_1^{4}$};

\draw[fill,red] (1,2) circle [radius=0.1];
\node[above left] at (1,2) {$\mathsf{X}_1^{3}$};

\draw[rotate around={135:(4,2)}] (4,2) ellipse (0.6cm and 2cm);
\draw[->, very thick,] (3,1) to [out=45, in=135] (4,1);
\draw[->, very thick,] (4,2) to [out=45, in=135] (5,2);
\draw[->, very thick,] (5,3) to [out=45, in=135] (6,3);

\draw[->, very thick] (1,2) to [out=45, in=135] (2,2);
\draw[very thick] (1.75,2) to (1.25,2.5);
\draw[very thick] (1.25,2) to (1.75,2.5);

\node[below] at (0,0) {$\star$};
\node[below] at (1,0) {$\star+1$};
\node[below] at (2,0) {$\star+2$};
\node[below] at (3,0) {$\star+3$};
\node[below] at (4,0) {$\star+4$};
\node[below] at (5,0) {$\star+5$};
\node[below] at (6,0) {$\star+6$};

\end{tikzpicture}
\caption{A configuration of particles in $\mathsf{GT}_4$. If the clock of the particle labelled $\mathsf{X}_1^{3}$ rings, which happens at rate $\lambda(\star+1)$, then the move is blocked since interlacing with $\mathsf{X}_1^{2}$ would be violated. On the other hand, if the clock of the particle $\mathsf{X}_2^{2}$ rings, which happens at rate $\lambda(\star+3)$, then it jumps to the right by one and instantaneously pushes both $\mathsf{X}_3^{3}$ and $\mathsf{X}_4^{4}$ to the right by one as well, for otherwise the interlacing would break.}\label{GelfandTsetlinDynamics}
\end{figure}

\begin{rmk}\label{EdgeSystemsRemark} \textsf{Inhomogeneous Boson and PushTASEP}. It is easy to see that the particle systems at the left $\left(\left(\mathsf{X}^{1}_1(t),\mathsf{X}^{2}_1(t),\dots,\mathsf{X}^{N}_1(t)\right);t\ge 0\right)$ and right $\left(\left(\mathsf{X}^{1}_1(t),\mathsf{X}^{2}_2(t),\dots,\mathsf{X}^{N}_N(t)\right);t\ge 0\right)$ edge respectively in the $\mathsf{GT}_N$-valued dynamics of Definition \ref{BFDynamics} enjoy an autonomous Markovian evolution. 

The left edge process is called the inhomogeneous TAZRP (totally asymmetric zero range process) or Boson particle system, see \cite{MacdonaldProcesses}, \cite{StochasticVertex}, \cite{qBoson}. In particular, in \cite{qBoson} a contour integral expression (for a $q$-deformation of the model) is obtained for its transition probabilities. The fact that, as we shall also see in the sequel, the distribution of particles at fixed time $t\ge 0$ is a marginal of a determinantal point process (with explicit kernel) is essentially contained in the results of \cite{KnizelPetrovSaenz} (which makes use of different methods). 

The right edge particle system is called inhomogeneous PushTASEP, see \cite{BorodinFerrari}, \cite{BorodinFerrariPushASEP} and also \cite{BorodinPetrov} for a $q$-deformation of the homogeneous case. The fact that this particle system has an underlying determinantal structure is new, but is also obtained in the independent work of Petrov \cite{PetrovInhomogeneousPushTASEP} that uses different methods.

\end{rmk}

\begin{rmk}\label{LevelInhomogeneousRemark1} \textsf{A generalization of the model}. Borodin and Ferrari in fact considered a more general model where all particles at level $k$ jump at rate $\beta_k$ (independent of their position). Using the techniques developed in this paper we can study the following model that allows for level inhomogeneities as well \footnote{We believe (but do not have a rigorous argument) that this is the most general model of continuous time Borodin-Ferrari dynamics (namely the particle interactions being as in Definition \ref{BFDynamics}) with determinantal correlations and which can be treated with the methods developed here. It is plausible however that if one considers discrete time dynamics that there exists an even more general model in inhomogeneous space involving geometric or Bernoulli jumps, as is the case in the homogeneous setting, see \cite{BorodinFerrari}.}. Let $\{\alpha_i \}_{i\ge 1}$ be a sequence of numbers such that:
\begin{align*}
0<\inf_{i\ge 1}\inf_{x\ge0}\left(\alpha_i+\lambda(x)\right)\le \sup_{i\ge 1}\sup_{x\ge0}\left(\alpha_i+\lambda(x)\right)<\infty.
\end{align*}
The dynamics are as in Definition \ref{BFDynamics} with the modification that each of the particles at level $k$ jumps at rate $\alpha_k+\lambda(\star)$ depending on its position $\star$. Since our main motivation in this work is the introduction of spatial inhomogeneities we will only consider the level homogeneous case of Definition \ref{BFDynamics} in detail. However, in the sequence of Remarks \ref{LevelInhomogeneousRemark2}, \ref{LevelInhomogeneousRemark3}, \ref{LevelInhomogeneousRemark4}, \ref{LevelInhomogeneousRemark5} we indicate the essential modifications required at each stage of the argument to study the more general model.
\end{rmk}

Observe that, for any $n\le N$ the process described in Definition \ref{BFDynamics} restricted to $\mathsf{GT}_n$ is an autonomous Markov process. We consider the natural projections:
\begin{align*}
\pi_N^{N+1}:\mathsf{GT}_{N+1}\longrightarrow \mathsf{GT}_{N},\ \forall N\ge 1,
\end{align*}
forgetting the top row $x^{(N+1)}$ and we write:
\begin{align*}
\mathsf{GT}_{\infty}=\underset{\leftarrow}{\lim}\mathsf{GT}_N
\end{align*}
for the corresponding projective limit, consisting of infinite Gelfand-Tsetlin patterns. We say that $\{\mathsf{M}_N\}_{N\ge 1}$ is a consistent sequence of distributions on $\{\mathsf{GT}_N \}_{N\ge 1}$ if:
\begin{align*}
\left(\pi_N^{N+1}\right)_*\mathsf{M}_{N+1}=\mathsf{M}_N, \ \forall N\ge 1.
\end{align*}
Suppose we are given such a consistent sequence of distributions  $\{\mathsf{M}_N\}_{N\ge 1}$. Then, by construction since the projections on any sub-pattern are autonomous, the processes $\left(\mathsf{X}_N\left(t;\mathsf{M}_N\right);t\ge 0\right)_{N\ge 1}$ are consistent as well:
\begin{align*}
\left(\pi_N^{N+1}\right)_*\mathsf{Law}\left[\mathsf{X}_{N+1}(t;\mathsf{M}_{N+1})\right]=\mathsf{Law}\left[\mathsf{X}_N(t;\mathsf{M}_N)\right],  \ \ \forall t\ge 0, \ \forall N\ge 1,
\end{align*}
and we can correctly define $\left(\mathsf{X}_\infty\left(t;\{\mathsf{M}_N\}_{N\ge 1}\right);t\ge 0\right)$, the corresponding process on $\mathsf{GT}_{\infty}$.

Now, we will be mainly concerned with the so-called densely packed initial conditions $\{\mathsf{M}_N^{\mathsf{dp}}\}_{N\ge 1}$ defined as follows:
\begin{align*}
\mathsf{M}_N^{\mathsf{dp}}=\delta_{\left(0\prec (0,1)\prec (0,1,2)\prec \cdots \prec(0,1,\dots,N-1)\right)}.
\end{align*}
Clearly $\{\mathsf{M}_N^{\mathsf{dp}}\}_{N\ge 1}$ is a consistent sequence of distributions. We simply write $\left(\mathsf{X}_\infty(t);t\ge0\right)$ for the corresponding process on $\mathsf{GT}_{\infty}$. 

Observe that, $\mathsf{X}_\infty(t)$ for any $t\ge 0$ gives rise to a random point process on $\mathbb{N}\times \mathbb{Z}_+$ which we denote by $\mathsf{P}^t_{\infty}$. We will use the notation $z=(n,x)$ to denote the location of a particle, with $n$ being the level/height/vertical position while $x$ is the horizontal position. Finally, it will be convenient to introduce the following functions, which will play a key role in the subsequent analysis.

\begin{defn} For $x \in \mathbb{Z}_+$, we define:
\begin{align}
\psi_x(w)=\psi_x(w;\lambda)=\prod_{k=0}^{x}\frac{\lambda(k)}{\lambda(k)-w}, \ p_x(w)=p_x(w;\lambda)=\prod_{k=0}^{x-1}\frac{\lambda(k)-w}{\lambda(k)}, p_0(w)\equiv 1.
\end{align}
\end{defn}

\begin{rmk}
Clearly $\psi_x(w;\lambda)=1/ p_{x+1}(w;\lambda)$ but it is preferable to think of them as two distinct families of functions. Observe that $p_x(w)$ is a polynomial of degree $x$ and $p_x(0)=1$.
\end{rmk}

We have then arrived at our main result.

\begin{thm}\label{MainTheorem} Let $\lambda:\mathbb{Z}_+\to (0,\infty)$ satisfy $(\mathsf{UB})$. Consider the point process $\mathsf{P}^t_{\infty}$ on $\mathbb{N}\times \mathbb{Z}_+$ obtained from running the dynamics of Definition \ref{BFDynamics} for time $t\ge0$ starting from the densely packed initial condition. Then for all $t\ge 0$, $\mathsf{P}^t_{\infty}$ has determinantal correlation functions. Namely, for any $k\ge 1$ and distinct points $z_1=(n_1,x_1),\dots,z_k=(n_k,x_k)\in \mathbb{N}\times \mathbb{Z}_+$:
\begin{align}
\mathsf{P}^t_{\infty}\left(\textnormal{there exist particles at locations } z_1,\dots,z_k\right)=\det\left[\mathsf{K}_t(z_i,z_j)\right]_{i,j=1}^k,
\end{align}
where the correlation kernel $\mathsf{K}_{t}(\cdot,\cdot;\cdot,\cdot)$ is explicitly given by:
\begin{align}
\mathsf{K}_{t}\left(n_1,x_1;n_2,x_2\right)=&\frac{1}{\lambda(x_1)} \frac{1}{2\pi \i} \oint_{\mathsf{C_\lambda}}\psi_{x_1}(w)\frac{p_{x_2}(w)}{w^{n_2-n_1}}dw\mathbf{1}(n_2>n_1)\\
&-\frac{1}{\lambda(x_1)}\frac{1}{(2\pi \i)^2}\oint_{\mathsf{C}_\lambda}dw\oint_{\mathsf{C}_0}du\psi_{x_1}(w)p_{x_2}(u)e^{-t(w-u)}\frac{w^{n_1-n_2}}{u^{n_2}}\frac{w^{n_2}-u^{n_2}}{w-u}
\end{align}
where $\mathsf{C}_{\lambda}$ is a counter clockwise contour encircling $0$ and the points $\{\lambda(x) \}_{x\ge 0}$ while $\mathsf{C}_0$ is a small counter clockwise contour around $0$ as in Figure \ref{FigureContours}.
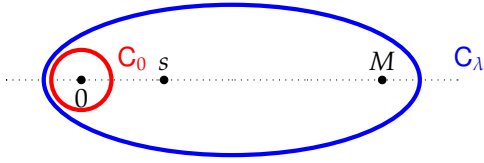
\begin{figure}
\captionsetup{singlelinecheck = false, justification=justified}
\begin{tikzpicture}
             \draw[dotted] (1,0) grid (7,0);
             \node[below] at (2,0) {$0$};
             \draw[fill] (2,0) circle [radius=0.05];
              \draw[ultra thick, red] (2,0) circle [radius=0.4];
              \node[above left,red] at (3,0) {$\mathsf{C}_{0}$};
              \node[above] at (3.1,0) {$s$};
              \draw[fill] (3.1,0) circle [radius=0.05];
              \node[above] at (6,0) {$M$};
              \draw[fill] (6,0) circle [radius=0.05];
              \node[above left,blue] at (7.5,0) {$\mathsf{C}_{\lambda}$};
             \draw[ultra thick,blue] (4,0) ellipse (2.5cm and 1cm);

 \end{tikzpicture}
 \caption{A possible choice of the contours $\mathsf{C}_0$ and $\mathsf{C}_{\lambda}$.}\label{FigureContours}
 \end{figure}
\end{thm}

\begin{rmk}
In the homogeneous case, $\lambda(\cdot)\equiv 1$, the correlation kernel in Theorem \ref{MainTheorem} above reduces to a kernel from \cite{BorodinFerrari} (there is a number of different kernels in \cite{BorodinFerrari} which give rise to equivalent determinantal point processes). To see this, it is most convenient to compare the expression for the kernel in Proposition 4.2 of \cite{BorodinFerrari} with the formulae from Section 3 herein (in particular see the expressions in display (4.4) in \cite{BorodinFerrari} and Lemmas \ref{AuxLemma1}, \ref{AuxLemma2} and display (\ref{PhiFunctionDefinition}) in this paper).
\end{rmk}

\subsection{Intermediate results and strategy of proof}

The proof of Theorem \ref{MainTheorem} essentially splits into two parts that we now elaborate on. Firstly showing that under certain initial conditions $\mathsf{M}_N$ on $\mathsf{GT}_N$ that we call Gibbs, of which the densely packed is a special case, the distribution of $\mathsf{X}_N\left(t;\mathsf{M}_N\right)$ for any $t\ge0$ is explicit (and again of Gibbs form), see Proposition \ref{ConsistentMultilevelProp}. By applying an extension of the famous Eynard-Mehta theorem \cite{EynardMehta} to interlacing particle systems, see \cite{BorodinRains}, \cite{BorodinFerrariPrahoferSasamoto}, it is then fairly standard, see Section 3.1, that under some (rather general) Gibbs initial conditions the point process associated to $\mathsf{X}_N\left(t;\mathsf{M}_N\right)$ has determinantal correlation functions (with a not yet explicit kernel).

A key ingredient in the argument for this first part is played by a remarkable block determinant kernel given in Definition \ref*{TwoLevelKM} that we call the two-level Karlin-McGregor formula (for the original, single level, Karlin-McGregor formula see \cite{KarlinMcGregorCoincidence}). This terminology is due to the fact that, as we see in Section 2 and in particular Theorem \ref{MarkovFunctionTheorem}, this provides a coupling between two Karlin-McGregor semigroups so that the corresponding processes interlace.

An instance of such a formula, in the context of Brownian motions, was first discovered by Warren in \cite{Warren}. Later it was understood that it can be further developed to include general one-dimensional diffusions and birth and death chains and this was achieved in \cite{InterlacingDiffusions} and \cite{RandomGrowthKarlinMcGregor} respectively. Part of our motivation for this paper was to enlarge the class of models for which such an underlying structure has been shown to exist\footnote{Currently this class has been shown to include essentially all examples of Borodin-Ferrari type dynamics in continuous time with determinantal correlations studied in the literature (both in continuous and discrete space, see \cite{InterlacingDiffusions} and \cite{RandomGrowthKarlinMcGregor} respectively). We expect such a formula to exist in the fully discrete setting (both time and space) as well and we plan to investigate this in future work. In this case, the Borodin-Ferrari dynamics give rise to shuffling algorithms for sampling random boxed plane partitions \cite{BorodinGorin},\cite{BorodinGorinRains} or domino tilings \cite{Nordenstam}, \cite{BorodinFerrariTilings}.}.

We note that, a crucial role in the previous works \cite{InterlacingDiffusions}, \cite{RandomGrowthKarlinMcGregor} was played by the notion of Siegmund duality, see \cite{Siegmund} (also \cite{KuanDuality1}, \cite{KuanDuality2} for other uses of duality in integrable probability) alongside with reversibility and their absence in the present setting is an important conceptual difference \footnote{In particular the statements of the results look different to the ones in \cite{InterlacingDiffusions} and \cite{RandomGrowthKarlinMcGregor}. More precisely, in \cite{InterlacingDiffusions} and \cite{RandomGrowthKarlinMcGregor} we obtain couplings between a Karlin-McGregor semigroup associated to a diffusion (or birth and death chain) and the one associated to (a Doob transform of) its Siegmund dual (which is in general different to the original process, except for a smaller sub-class of self-dual ones). In the present paper the couplings are between two Karlin-McGregor semigroup associated to a general pure-birth chain (and Doob transforms thereof), see Section 2 for more details.}. However, as it turns out an analogue of this duality suitable for our purposes does exist and is proven in Lemma \ref{DualityLemma}.

Moreover in Section 2 we prove some results on general non-intersecting pure-birth chains that generalize the Charlier process \cite{NonCollidingWalks}, \cite{BorodinFerrari}, the discrete analogue of Dyson's Brownian motion \cite{DysonBrownian}. In particular, we construct harmonic functions (and also more general eigenfunctions) that are given as determinants with explicit entries.

Finally, we should mention that all of these formulae have in some sense their origin in the study of coalescing stochastic flows of diffusions and Markov chains, see \cite{LeJan}. Recently a general abstract theory has been developed for constructing couplings between intertwined Markov processes based on random maps and coalescing flows, see \cite{Miclo}. It would be interesting to understand to what extent this is related to the present constructions and more generally to analogous couplings in integrable probability.

The second part of the proof involves the solution of a certain biorthogonalization problem which gives the explicit form of the correlation kernel $\mathsf{K}_t$ and is performed in Section 3. An important difference to the works \cite{BorodinFerrari}, \cite{BorodinKuan}, \cite{BorodinFerrariPrahoferSasamoto}, \cite{BorodinKuanPlancherel} where corresponding biorthogonalization problems were analysed is that the functions involved in the current problem are not translationally invariant in the spatial variable. This is where we make use of the functions $\{\psi_x \}_{x\in \mathbb{Z}_+}$ and $\{p_x\}_{x\in \mathbb{Z}_+}$ which arise in the spectral theory of a general pure-birth chain (see display (\ref{SpectralExpansion}) for the spectral expansion of the transition kernel of the chain in terms of them). These provide both intuition and also make most of the (otherwise quite tedious) computations rather neat, see in particular the sequence of Lemmas \ref{AuxLemma1}, \ref{AuxLemma2}, \ref{AuxLemma3}, \ref{AuxLemma4}, \ref{AuxLemma5} and their proofs.

\paragraph{Acknowledgements.} I would like to thank Maurice Duits, Patrik Ferrari and Jon Warren for some early discussions which motivated the problems studied in this paper. Moreover, I am very grateful to Leonid Petrov for sending me his preprint and for some interesting questions and remarks. I would also like to thank Alexei Borodin for some very interesting suggestions and pointers to the literature. Finally, I am very grateful to an anonymous referee for a careful reading of the paper and some very useful suggestions and remarks which have improved the presentation. The research described here was supported by ERC Advanced Grant  740900 (LogCorRM).

\section{Inhomogeneous space push-block dynamics}

\subsection{The one dimensional chain}
We define the forward and backwards discrete derivatives
\begin{align*}
\left[\nabla^+f\right](x)=f(x+1)-f(x) \ , \ \left[\nabla^-f\right](x)=f(x-1)-f(x).
\end{align*}
We define the following pure-birth chain which is the basic building block of our construction: this is a continuous time Markov chain on $\mathbb{Z}_+$ which when at site $x$ jumps to site $x+1$ with rate $\lambda(x)$. Its generator is then given by (the subscript indicates the variable on which it is acting):
\begin{align}
\mathsf{L}=\mathsf{L}_x^{\lambda}=\lambda(x)\nabla^+_x.
\end{align}
Due to \textsf{(UB)} non-explosiveness and thus existence and uniqueness of the pure-birth process is immediate (simply compare with a Poisson process with constant rate $M$). We write $\left(e^{t\mathsf{L}};t\ge 0\right)$ for the corresponding Markov semigroup and  abusing notation we also write $e^{t\mathsf{L}}(x,y)$ for its transition density, namely the probability a Markov chain with generator $\mathsf{L}$ goes from site $x$ to $y$ in time $t$. This is the unique solution to both the Kolmogorov backwards and forwards equations, see Section 2.6 of \cite{Norris}. The backwards equation, which we make use of here, reads as follows, for $t>0,x,y \in \mathbb{Z}_+$:
 \begin{align*}
\frac{d}{dt}u_t(x,y)&=\mathsf{L}_xu_t(x,y),\\ u_0(x,y)&=\mathbf{1}(x=y).
 \end{align*}
It is an elementary probabilistic argument that $e^{t\mathsf{L}}(x,y)$ is explicit, see Section 2 of \cite{qBoson} (or alternatively simply check that the expression below solves the Kolmogorov equation). We have the following spectral expansion for it, see display (2.1a) of \cite{qBoson}:
\begin{align}\label{SpectralExpansion}
e^{t\mathsf{L}}(x,y)=-\frac{1}{\lambda(y)}\frac{1}{2\pi \i} \oint_{\mathsf{C}_{\lambda}}\psi_y(w)p_x(w)e^{-tw}dw, \ t\ge 0, x,y\in \mathbb{Z}_+.
\end{align}

\begin{rmk}
Here we can simply pick any counter clockwise contour which encircles the points $\{\lambda(x)\}_{x\ge0}$ and not necessarily $0$ as well. The fact that $\mathsf{C}_{\lambda}$ encircles $0$ will be useful in the computation of the correlation kernel later on.
\end{rmk} 

Throughout the paper we use the notation $\mathbf{1}_{[\![ 0,y]\!]}(\cdot)$ for the indicator function of the set $\{0,1,2,\dots,y \}$. We have the following key relation for the transition density $e^{t\mathsf{L}}(x,y)$.

\begin{lem}\label{DualityLemma}
Let $x,y \in \mathbb{Z}_+$ and $t\ge 0$ we have:
\begin{align*}
e^{t\mathsf{L}}(x,y)=-\frac{\lambda(x)}{\lambda(y)}\nabla^+_xe^{t\mathsf{L}} \mathbf{1}_{[\![ 0,y]\!]}(x).
\end{align*}
\end{lem}

\begin{proof}
We let
\begin{align*}
s_t(x,y)=-\frac{\lambda(x)}{\lambda(y)}\nabla^+_xe^{t\mathsf{L}}\mathbf{1}_{[\![ 0,y]\!]}(x)
\end{align*}
and show that this solves the Kolmogorov backwards equation. By uniqueness the statement follows. The $t=0$ initial condition follows from:
\begin{align*}
\nabla_x^+\mathbf{1}(x\le y)=\mathbf{1}(x+1\le y)-\mathbf{1}(x\le y)=-\mathbf{1}(x=y).
\end{align*}
Finally,
\begin{align*}
\frac{d}{dt}s_t(x,y)=-\frac{\lambda(x)}{\lambda(y)}\nabla^+_x\frac{d}{dt}e^{t\mathsf{L}}\mathbf{1}_{[\![0,y]\!]}(x)=-\frac{\lambda(x)}{\lambda(y)}\nabla^+_x\lambda(x)\nabla^+_xe^{t\mathsf{L}}\mathbf{1}_{[\![ 0,y]\!]}(x)=\mathsf{L}_xs_t(x,y)
\end{align*}
as required.
\end{proof}

\begin{rmk}
It is also possible to prove Lemma \ref{DualityLemma} using the explicit spectral expansion (\ref{SpectralExpansion}), the argument is similar to the one in the proof of Lemma \ref{AuxLemma1}.
\end{rmk}

\subsection{Two level dynamics}
We begin with a classical definition due to Karlin and McGregor \cite{KarlinMcGregorCoincidence}.  
\begin{defn} The Karlin-McGregor sub-Markov semigroup on $\mathbb{W}^N$ associated to a pure-birth chain with generator $\mathsf{L}$ is defined by its transition density given by, for $t\ge 0, x,y \in \mathbb{W}^N$: 
\begin{align*}
\mathcal{P}_t^N\left(x,y\right)=\det\left(e^{t\mathsf{L}}(x_i,y_j)\right)_{i,j=1}^N.
\end{align*}
\end{defn}
This semigroup has the following probabilistic interpretation: it corresponds to $N$ independent copies of a chain with generator $\mathsf{L}$ killed when they intersect, see \cite{KarlinMcGregorCoincidence}, \cite{Karlin}. A conditioned upon non-intersection version of this semigroup will govern the dynamics on projections on single levels of our interlacing arrays, see Theorem \ref{MarkovFunctionTheorem} and Proposition \ref{ConsistentMultilevelProp} below.

We now move on to study two-level dynamics. We first consider the space of pairs $(y,x)$ which interlace:
\begin{align*}
\mathbb{W}^{N,N+1}&=\{(y,x)=(y_1,\dots,y_N,x_1,\dots,x_{N+1})\in \mathbb{Z}_+^{2N+1}: x_1 \le y_1 < x_2 \le \cdots \le y_N< x_{N+1}\}.
\end{align*}
We have the following key definition that we call the two-level Karlin-McGregor formula, since as we shall see in the sequel this provides a coupling for two Karlin-McGregor semigroups, so that the corresponding processes interlace.

\begin{defn}\label{TwoLevelKM}
For $(y,x),(y',x')\in \mathbb{W}^{N,N+1}$ and $t\ge0$, define $ \mathsf{U}_t^{N,N+1}\left[(y,x),(y',x')\right]$ by the following $(2N+1)\times (2N+1)$ block matrix determinant:
\begin{align}
 \mathsf{U}_t^{N,N+1}\left[(y,x),(y',x')\right]=\det\
  \begin{pmatrix}
 \mathsf{A}_t(x,x') & \mathsf{B}_t(x,y')\\
   \mathsf{C}_t(y,x') & \mathsf{D}_t(y,y') 
  \end{pmatrix},
\end{align}
where the matrices $\mathsf{A}_t,\mathsf{B}_t, \mathsf{C}_t, \mathsf{D}_t$ of sizes $(N+1)\times (N+1)$, $(N+1)\times N$, $N\times (N+1)$ and $N\times N$ respectively are given by:
\begin{align*}
\mathsf{A}_t(x,x')_{ij} &=-\nabla^-_{x'_j}e^{t\mathsf{L}} \mathbf{1}_{[\![ 0,x_j']\!]} (x_i)= e^{t\mathsf{L}}(x_i,x_j'),  \\
\mathsf{B}_t(x,y')_{ij}&=\lambda(y'_j)^{-1}(e^{t\mathsf{L}} \mathbf{1}_{[\![ 0,y_j']\!]}(x_i) -\mathbf{1}(j\ge i)),\\
\mathsf{C}_t(y,x')_{ij}&=\lambda(y_i)\nabla^+_{y_i}\nabla^-_{x'_j}e^{t\mathsf{L}} \mathbf{1}_{[\![ 0,x_j']\!]}(y_i),\\
\mathsf{D}_t(y,y')_{ij}&=-\frac{\lambda(y_i)}{\lambda(y'_j)}\nabla^+_{y_i}e^{t\mathsf{L}}\mathbf{1}_{[\![ 0,y_j']\!]}(y_i)=e^{t\mathsf{L}}(y_i,y_j').
\end{align*}
Observe that, the equivalence of the two representations for $\mathsf{A}_t$ is by definition while for $\mathsf{D}_t$ is due to Lemma \ref{DualityLemma}.
\end{defn}

We begin our study of $\mathsf{U}_t$ by proving some of its basic properties:

\begin{lem}\label{SubstochasticLemma}
For all $t\ge 0$ the kernel $\mathsf{U}_t^{N,N+1}$ satisfies: $\mathsf{U}_t^{N,N+1}\left[(y,x),(y',x')\right]\ge 0$, for $(y,x),(y',x')\in \mathbb{W}^{N,N+1}$ and $\mathsf{U}_t^{N,N+1}\mathbf{1}\le \mathbf{1}$.
\end{lem}

\begin{proof}
We prove positivity first. A direct verification from Definition \ref{TwoLevelKM} appears to be hard (although it would be interesting to have one). Instead, we give a simple probabilistic argument. 

Let $\left((\mathcal{S}_t(x_1),\dots,\mathcal{S}_t(x_m));t\ge 0\right)$ be a system of $m$ independent chains with generator $\mathsf{L}$, starting from $(x_1,\dots,x_m) \in \mathbb{W}^m$, and which coalesce and move together once any two of them meet. We denote their law by $\mathbb{P}$. Let $z,z' \in \mathbb{W}^m$ and $t\ge 0$. Then, we have:
\begin{align}\label{CoalescingRepresentation}
\mathbb{P}\big(\mathcal{S}_{t}(z_i)\le z_i' \ , \ \textnormal{for} \ 1 \le i \le m\big)=\det \big(e^{t\mathsf{L}}\mathbf{1}_{[\![ 0,z_j']\!]}(z_i)-\textbf{1}(i<j)\big)_{i,j=1}^m.
\end{align}
The claim is a consequence of the Karlin-McGregor formula, see Proposition 2.5 in \cite{RandomGrowthKarlinMcGregor} for a proof in a completely analogous setting. Observe that (\ref{CoalescingRepresentation}) allows to give the following probabilistic representation for $\mathsf{U}_t^{N,N+1}$ (by applying discrete derivatives to the RHS of (\ref{CoalescingRepresentation}) we match with the expression from Definition \ref{TwoLevelKM}):
\begin{align*}
& \mathsf{U}_t^{N,N+1}[(y,x),(y',x')]=\\
&=\frac{\prod_{i=1}^{N}\lambda(y_i)}{\prod_{i=1}^{N}\lambda(y'_i)}(-1)^N\nabla^+_{y_1}\cdots\nabla^+_{y_N}(-1)^{N+1}\nabla^-_{x'_1}\cdots\nabla^-_{x'_{N+1}}\mathbb{P}\big(\mathcal{S}_{t}(x_i)\le x_i',\mathcal{S}_{t}(y_j)\le y_j' \ \ \textnormal{for all} \ \ i,j \big).
\end{align*}
Positivity of $\mathsf{U}_t^{N,N+1}$ is then a consequence of the fact that the events
\begin{align*}
\big\{\mathcal{S}_{t}(x_i)\le x_i',\mathcal{S}_{t}(y_j)\le y_j' \ \ \textnormal{for all} \ \ i,j  \big\}
\end{align*}
are increasing both as the variables $y_i$ decrease and as the variables $x'_j$ increase.

Finally, we need to prove that for any $(y,x) \in \mathbb{W}^{N,N+1}$ and $t\ge 0$:
\begin{align*}
\sum_{(y',x') \in \mathbb{W}^{N,N+1} }^{}\mathsf{U}_t^{N,N+1}\left[(y,x),(y',x')\right] \le 1.
\end{align*}
We claim that for any $(y,x) \in \mathbb{W}^{N,N+1}$ and $t\ge 0$:
\begin{align*}
\sum_{\{x': (y',x') \in \mathbb{W}^{N,N+1} \}}^{}\mathsf{U}_t^{N,N+1}\left[(y,x),(y',x')\right]=\det \left(\mathsf{D}_t(y_i,y_j')\right)_{i,j=1}^N=\mathcal{P}_t^N(y,y').
\end{align*}
Since $\left(\mathcal{P}_t^N;t\ge 0\right)$ is sub-Markov the statement of the proposition follows. Now in order to prove the claim we take the sum $\sum_{\{x': (y',x') \in \mathbb{W}^{N,N+1} \}}^{}$ in the explicit form of the kernel from Definition \ref{TwoLevelKM} and use multilinearity of the determinant. Then, the claim follows from the relations below:
\begin{align*}
\sum_{x_j'=y'_{j-1}+1}^{y'_j} \mathsf{A}_t\left(x,x'\right)_{ij}&=e^{t\mathsf{L}}\mathbf{1}_{[\![ 0,y_j']\!]}(x_i)-e^{t\mathsf{L}}\mathbf{1}_{[\![ 0,y_{j-1}']\!]}(x_i),\\
\sum_{x_j'=y'_{j-1}+1}^{y'_j}\mathsf{C}_{t}(y,x')_{ij} &=-\lambda(y_i)\nabla^+_{y_i} e^{t\mathsf{L}} \mathbf{1}_{[\![ 0,y_j']\!]}(y_i)+\lambda(y_{i})\nabla^+_{y_i} e^{t\mathsf{L}} \mathbf{1}_{[\![ 0,y_{j-1}']\!]}(y_i)
\end{align*}
and simple row-column operations.
\end{proof}

We now introduce the inhomogeneous two-level dynamics:

\begin{defn}\label{TwoLevelDynamics}
\textsf{Two level inhomogeneous push-block dynamics}. This is the continuous time Markov chain $\left(\left(Y(t),X(t)\right);t\ge 0\right)$ in $\mathbb{W}^{N,N+1}$, with possibly finite lifetime, described as follows. Each of the $2N+1$ particles evolves as an independent chain with generator $\mathsf{L}$ subject to the following interactions. The $Y$-particles evolve autonomously. When a potential move by the $X$-particles would break the interlacing it is blocked, see Figure \ref{figureXYinteraction2}. While if a potential move by the $Y$-particles would break the interlacing then the corresponding $X$-particle is pushed to the right by one, see Figure \ref{figureXYinteraction1}. The Markov chain is killed when two $Y$-particles collide, at the stopping time:
\begin{align*}
\tau=\inf\{t>0:\exists \ 1\le i<j \le N \ ,\textnormal{ such that } Y_i(t)=Y_j(t)\}.
\end{align*}
\end{defn}

\begin{figure}
\captionsetup{singlelinecheck = false, justification=justified}
\begin{tikzpicture}
\draw[dotted] (1,0) grid (4,1);
\draw[fill,red] (1,1) circle [radius=0.1];
\node[above left] at (1,1) {$x_{i}$};

\draw[fill] (1,0) circle [radius=0.1];
\node[above right] at (1,0) {$y_i$};

\draw[->, very thick] (1,1) to [out=60, in=120] (2,1);
\draw[very thick] (1.75,1) to (1.25,1.5);
\draw[very thick] (1.25,1) to (1.75,1.5);
\node[below] at (1,0) {$\star$};
\node[below] at (2,0) {$\star+1$};
\node[below] at (3,0) {$\star+2$};
\node[below] at (4,0) {$\star+3$};

\draw[->,very thick] (5,0.5) to (7,0.5);

\draw[dotted] (8,0) grid (11,1);
\draw[fill] (8,1) circle [radius=0.1];
\node[above right] at (8,1) {$x_{i}$};

\draw[fill] (8,0) circle [radius=0.1];
\node[above right] at (8,0) {$y_i$};

\node[below] at (8,0) {$\star$};
\node[below] at (9,0) {$\star+1$};
\node[below] at (10,0) {$\star+2$};
\node[below] at (11,0) {$\star+3$};

\end{tikzpicture}

\caption{($\sf{Blocking}$) A jump of $x_{i}$ is blocked by $y_{i}$ so that the interlacing is maintained. Here, the clock of $x_{i}$ rings with rate $\lambda(\star)$.}\label{figureXYinteraction2}
\end{figure}
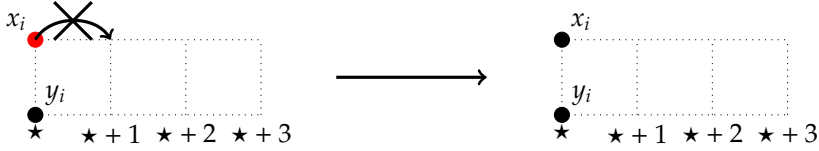

\begin{figure}
\captionsetup{singlelinecheck = false, justification=justified}
\begin{tikzpicture}
\draw[dotted] (1,0) grid (4,1);
\draw[fill] (2,1) circle [radius=0.1];
\node[above right] at (2,1) {$x_{i+1}$};

\draw[fill,red] (1,0) circle [radius=0.1];
\node[above right] at (1,0.15) {$y_i$};

\draw[rotate around={135:(1.5,0.5)}] (1.5,0.5) ellipse (0.6cm and 1cm);
\draw[->, very thick,] (1,0) to [out=60, in=120] (2,0);
\node[below] at (1,0) {$\star$};
\node[below] at (2,0) {$\star+1$};
\node[below] at (3,0) {$\star+2$};
\node[below] at (4,0) {$\star+3$};

\draw[->,very thick] (5,0.5) to (7,0.5);

\draw[dotted] (8,0) grid (11,1);
\draw[fill] (10,1) circle [radius=0.1];
\node[above right] at (10,1) {$x_{i+1}$};

\draw[fill] (9,0) circle [radius=0.1];
\node[above right] at (9,0) {$y_i$};

\node[below] at (8,0) {$\star$};
\node[below] at (9,0) {$\star+1$};
\node[below] at (10,0) {$\star+2$};
\node[below] at (11,0) {$\star+3$};

\end{tikzpicture}
\caption{($\sf{Pushing}$) A jump of $y_i$ induces a simultaneous jump (pushes) of $x_{i+1}$ to the right so that the interlacing is maintained. Here, the clock of $y_i$ rings with rate $\lambda(\star)$.}\label{figureXYinteraction1}
\end{figure}
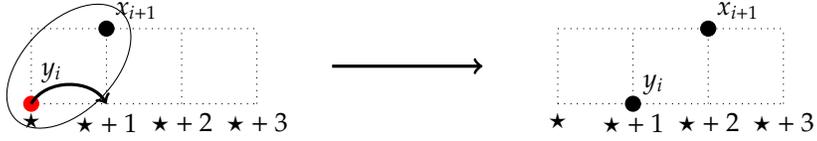

\begin{prop}
The block determinant kernel $\mathsf{U}_t^{N,N+1}$ forms the transition density for the dynamics in Definition \ref{TwoLevelDynamics}.
\end{prop}

\begin{proof}
We show that $\mathsf{U}_t^{N,N+1}$ solves the Kolmogorov's backward equation corresponding to the dynamics in Definition \ref{TwoLevelDynamics}. Uniqueness in the class of substochastic matrices (which $\mathsf{U}_t^{N,N+1}$ is a member of by Lemma \ref{SubstochasticLemma})  follows by a generic argument presented in a completely analogous setting in Section 3 in \cite{RandomGrowthKarlinMcGregor}, see also \cite{BorodinOlshanski}.

First, observe that we have the $t=0$ initial condition: 
\begin{align}
 \mathsf{U}_0^{N,N+1}\left[(y,x),(y',x')\right]=\mathbf{1}\left((y,x)=(y',x')\right), \ (y,x), (y',x') \in \mathbb{W}^{N,N+1}.
\end{align}
This follows directly from the form of $ \mathsf{U}_t^{N,N+1}\left[(y,x),(y',x')\right]$, by noting that as $t \downarrow 0$, the diagonal entries converge to $\mathbf{1} \left(x_i=x_i'\right),\mathbf{1}\left(y_i=y_i'\right)$, while all other contributions to the determinant vanish.

Moreover, observe that we have the Dirichlet boundary conditions when two $Y$-coordinates coincide:
\begin{align*}
\mathsf{U}_t^{N,N+1}\left[(y,x),(y',x')\right]=0, \ y_i=y_{i+1}.
\end{align*}

Moving on, note that (here we are abusing notation slightly by using the same notation for both the matrices and their scalar entries) we have the following, for any $x,x'\in \mathbb{Z}_+$ fixed and $t>0$:
\begin{align}
\frac{d}{dt}\mathsf{A}_t(x,x')=\mathsf{L}_x\mathsf{A}_t(x,x'),
\frac{d}{dt}\mathsf{B}_t(x,x')=\mathsf{L}_x\mathsf{B}_t(x,x'),\label{diffeq1}\\
\frac{d}{dt}\mathsf{C}_t(x,x')=\mathsf{L}_x\mathsf{C}_t(x,x'),
\frac{d}{dt}\mathsf{D}_t(x,x')=\mathsf{L}_x\mathsf{D}_t(x,x')\label{diffeq2}.
\end{align}
To see the relation for $\mathsf{C}_t$ observe that:
\begin{align*}
\frac{d}{dt}\mathsf{C}_t(x,y)&=\lambda(x)\nabla^+_{x}\nabla^-_{y}\frac{d}{dt}e^{t\mathsf{L}} \mathbf{1}_{[\![ 0,y]\!]}(x)=\lambda(x)\nabla^+_{x}\nabla^-_{y}\lambda(x)\nabla_x^+e^{t\mathsf{L}} \mathbf{1}_{[\![ 0,y]\!]}(x)\\
&=\lambda(x)\nabla^+_{x}\lambda(x)\nabla_x^+\nabla^-_{y}e^{t\mathsf{L}} \mathbf{1}_{[\![ 0,y]\!]}(x)=\mathsf{L}_x\mathsf{C}_t(x,y).
\end{align*}

We define $\mathsf{int}\left(\mathbb{W}^{N,N+1}\right)$ to be the set of all pairs $(y,x)\in \mathbb{W}^{N,N+1}$ which when any of the $x$ or $y$ coordinates is increased by $1$ they still belong to $\mathbb{W}^{N,N+1}$; namely the pairs $(y,x)\in \mathbb{W}^{N,N+1}$ so that $(y,x)+(\mathsf{e}_i,0),(y,x)+(0,\mathsf{e}_i)\in \mathbb{W}^{N,N+1}$ with $\mathsf{e}_i$ being the unit vector in the $i$-th coordinate. Observe that in $\mathsf{int}\left(\mathbb{W}^{N,N+1}\right)$ each of coordinates evolves as an independent chain with generator $\mathsf{L}$ which do not interact. Then, by the multilinearity of the determinant and relations (\ref{diffeq1}) and (\ref{diffeq2}) we obtain:
\begin{align*}
\frac{d}{dt}\mathsf{U}_t^{N,N+1}\left[(y,x),(y',x')\right]=\left(\sum_{i=1}^{N+1}\mathsf{L}_{x_i}+\sum_{i=1}^{N}\mathsf{L}_{y_i}\right)\mathsf{U}_t^{N,N+1}\left[(y,x),(y',x')\right], \ t>0, (y,x) \in \mathsf{int}\left(\mathbb{W}^{N,N+1}\right).
\end{align*}
It remains to deal with the interactions. We will only consider one blocking and one pushing case, as all others are entirely analogous. First, the blocking case with $x_1=y_1=x$. In order to ease notation and also make the gist of the simple argument transparent we further restrict our attention to the rows containing $x_1,y_1$. In fact, it is not hard to see that it suffices to consider the $2\times 2$ matrix determinant given by, with $x',y'\in \mathbb{Z}$ fixed:
\begin{align*}
\det\
 \begin{pmatrix}
\mathsf{A}_t(x,x') & \mathsf{B}_t(x,y')\\
  \mathsf{C}_t(x,x') & \mathsf{D}_t(x,y') 
 \end{pmatrix}.
\end{align*}
By taking the $\frac{d}{dt}$-derivative of the determinant, we obtain using (\ref{diffeq1}) and (\ref{diffeq2}):

\begin{align*}
\frac{d}{dt}\det\
 \begin{pmatrix}
\mathsf{A}_t(x,x') & \mathsf{B}_t(x,y')\\
  \mathsf{C}_t(x,x') & \mathsf{D}_t(x,y') 
 \end{pmatrix}=\lambda(x)\left[\det\
  \begin{pmatrix}
 \mathsf{A}_t(x+1,x') & \mathsf{B}_t(x+1,y')\\
   \mathsf{C}_t(x,x') & \mathsf{D}_t(x,y') 
  \end{pmatrix}-\det\
   \begin{pmatrix}
 \mathsf{A}_t(x,x') & \mathsf{B}_t(x,y')\\
    \mathsf{C}_t(x,x') & \mathsf{D}_t(x,y') 
   \end{pmatrix}\right]\\
   +\lambda(x)\left[\det\
      \begin{pmatrix}
     \mathsf{A}_t(x,x') & \mathsf{B}_t(x,y')\\
       \mathsf{C}_t(x+1,x') & \mathsf{D}_t(x+1,y') 
      \end{pmatrix}-\det\
       \begin{pmatrix}
     \mathsf{A}_t(x,x') & \mathsf{B}_t(x,y')\\
        \mathsf{C}_t(x,x') & \mathsf{D}_t(x,y') 
       \end{pmatrix}\right].       
\end{align*}
On the other hand, what we would like to have according to the dynamics in Definition \ref{TwoLevelDynamics} is simply the following:
\begin{align*}
\frac{d}{dt}\det\
 \begin{pmatrix}
\mathsf{A}_t(x,x') & \mathsf{B}_t(x,y')\\
  \mathsf{C}_t(x,x') & \mathsf{D}_t(x,y') 
 \end{pmatrix}=
    \lambda(x)\left[\det\
      \begin{pmatrix}
     \mathsf{A}_t(x,x') & \mathsf{B}_t(x,y')\\
       \mathsf{C}_t(x+1,x') & \mathsf{D}_t(x+1,y') 
      \end{pmatrix}-\det\
       \begin{pmatrix}
      \mathsf{A}_t(x,x') & \mathsf{B}_t(x,y')\\
        \mathsf{C}_t(x,x') & \mathsf{D}_t(x,y') 
       \end{pmatrix}\right].
\end{align*}
We thus must show that:
\begin{align}\label{equality1}
\det\
  \begin{pmatrix}
 \mathsf{A}_t(x+1,x') & \mathsf{B}_t(x+1,y')\\
   \mathsf{C}_t(x,x') & \mathsf{D}_t(x,y') 
  \end{pmatrix}=\det\
   \begin{pmatrix}
  \mathsf{A}_t(x,x') & \mathsf{B}_t(x,y')\\
    \mathsf{C}_t(x,x') & \mathsf{D}_t(x,y') 
   \end{pmatrix},
\end{align}
which corresponds to particle $x_1$ being blocked when $x_1=y_1$ and $x_1$ tries to jump (see the configuration in Figure \ref{figureXYinteraction2}). 
In order to obtain (\ref{equality1}) we shall work on the RHS. We multiply the second row by $-\lambda(x)^{-1}$ and add it to the first row to obtain
\begin{align*}
\mathsf{A}_t(x,x')-\lambda(x)^{-1}\mathsf{C}_t(x,x')&=-\nabla^-_{x'}e^{t\mathsf{L} }\mathbf{1}_{[\![ 0,x']\!]}(x)-\nabla^-_{x'}e^{t\mathsf{L} }\mathbf{1}_{[\![ 0,x']\!]}(x+1)+\nabla^-_{x'}e^{t\mathsf{L} }\mathbf{1}_{[\![ 0,x']\!]}(x)\\
&=-\nabla^-_{x'}e^{t\mathsf{L} }\mathbf{1}_{[\![ 0,x']\!]}(x+1)
=\mathsf{A}_t(x+1,x'),
\end{align*}
and analogously for the second column, which then gives us the LHS of (\ref{equality1}) as desired.

Similarly, we consider a pushing move with $y_1=x,x_2=x+1$ and $x',y'\in \mathbb{Z}_+$ fixed (see the configuration in Figure \ref{figureXYinteraction1}):
\begin{align*}
\det\
 \begin{pmatrix}
\mathsf{A}_t(x+1,x') & \mathsf{B}_t(x+1,y')\\
  \mathsf{C}_t(x,x') & \mathsf{D}_t(x,y') 
 \end{pmatrix}.
\end{align*}
We calculate using the relations (\ref{diffeq1}) and (\ref{diffeq2}):
\begin{align*}
&\frac{d}{dt}\det\
 \begin{pmatrix}
\mathsf{A}_t(x+1,x') & \mathsf{B}_t(x+1,y')\\
  \mathsf{C}_t(x,x') & \mathsf{D}_t(x,y') 
 \end{pmatrix}\\&=\lambda(x+1)\left[\det\
  \begin{pmatrix}
 \mathsf{A}_t(x+2,x') & \mathsf{B}_t(x+2,y')\\
   \mathsf{C}_t(x,x') & \mathsf{D}_t(x,y') 
  \end{pmatrix}-\det\
   \begin{pmatrix}
 \mathsf{A}_t(x+1,x') & \mathsf{B}_t(x+1,y')\\
    \mathsf{C}_t(x,x') & \mathsf{D}_t(x,y') 
   \end{pmatrix}\right]\\
   &+\lambda(x)\left[\det\
      \begin{pmatrix}
     \mathsf{A}_t(x+1,x') & \mathsf{B}_t(x+1,y')\\
       \mathsf{C}_t(x+1,x') & \mathsf{D}_t(x+1,y') 
      \end{pmatrix}-\det\
       \begin{pmatrix}
     \mathsf{A}_t(x+1,x') & \mathsf{B}_t(x+1,y')\\
        \mathsf{C}_t(x,x') & \mathsf{D}_t(x,y') 
       \end{pmatrix}\right].       
\end{align*}
From the dynamics in Definition \ref{TwoLevelDynamics} we need to have the following:
\begin{align*}
&\frac{d}{dt}\det\
 \begin{pmatrix}
\mathsf{A}_t(x+1,x') & \mathsf{B}_t(x+1,y')\\
  \mathsf{C}_t(x,x') & \mathsf{D}_t(x,y') 
 \end{pmatrix}\\&=\lambda(x+1)\left[\det\
  \begin{pmatrix}
 \mathsf{A}_t(x+2,x') & \mathsf{B}_t(x+2,y')\\
   \mathsf{C}_t(x,x') & \mathsf{D}_t(x,y') 
  \end{pmatrix}-\det\
   \begin{pmatrix}
 \mathsf{A}_t(x+1,x') & \mathsf{B}_t(x+1,y')\\
    \mathsf{C}_t(x,x') & \mathsf{D}_t(x,y') 
   \end{pmatrix}\right]\\
   &+\lambda(x)\left[\det\
      \begin{pmatrix}
     \mathsf{A}_t(x+2,x') & \mathsf{B}_t(x+2,y')\\
       \mathsf{C}_t(x+1,x') & \mathsf{D}_t(x+1,y') 
      \end{pmatrix}-\det\
       \begin{pmatrix}
     \mathsf{A}_t(x+1,x') & \mathsf{B}_t(x+1,y')\\
        \mathsf{C}_t(x,x') & \mathsf{D}_t(x,y') 
       \end{pmatrix}\right].       
\end{align*}
Hence, we need to show:
\begin{align*}
\det\
  \begin{pmatrix}
 \mathsf{A}_t(x+2,x') & \mathsf{B}_t(x+2,y')\\
   \mathsf{C}_t(x+1,x') & \mathsf{D}_t(x+1,y') 
  \end{pmatrix}=\det\
   \begin{pmatrix}
  \mathsf{A}_t(x+1,x') & \mathsf{B}_t(x+1,y')\\
    \mathsf{C}_t(x+1,x') & \mathsf{D}_t(x+1,y') 
   \end{pmatrix},
\end{align*}
which follows from (\ref{equality1}) after relabelling $x \to x+1$.

\end{proof}

We now need a couple of definitions whose purpose will be clear shortly.

\begin{defn}
We define the positive kernel $\Lambda_{N}^{N+1}$ from $\mathbb{W}^{N+1}$ to $\mathbb{W}^N$ by its density (with respect to counting measure):
\begin{align*}
\Lambda_N^{N+1}\left(x,y\right)=\prod_{i=1}^{N}\frac{1}{\lambda(y_i)}\mathbf{1}\left(y\prec x\right).
\end{align*}
Abusing notation, we can also view $\Lambda_N^{N+1}$ as a kernel from $\mathbb{W}^{N+1}$ to $\mathbb{W}^{N,N+1}$, in which case we write $\Lambda_N^{N+1}(x,(y,z))=\Lambda_N^{N+1}(x,y)$. Observe that, this is supported on elements $(y,z)\in \mathbb{W}^{N,N+1}$ such that $z\equiv x$.
\end{defn}

\begin{defn}
For any $N\ge 1$ and $x \in \mathbb{W}^N$ define the functions $\mathfrak{h}_N\left(x\right)=\mathfrak{h}_N\left(x;\lambda\right)$ recursively by, $\mathfrak{h}_1(x)\equiv 1$ and
\begin{align*}
\mathfrak{h}_{N+1}(x)=\left[\Lambda_N^{N+1}\mathfrak{h}_N\right](x).
\end{align*}
\end{defn}

The functions $\mathfrak{h}_N$ can in fact be written as determinants whose entries are defined recursively:
\begin{lem}\label{DeterminantRepresentationHarmonic}
Let $N\ge 1$ and $x \in \mathbb{W}^N$. Then,
\begin{align*}
\mathfrak{h}_N(x;\lambda)=\det \left(\mathfrak{I}_{i-1}(x_j;\lambda)\right)_{i,j=1}^N,
\end{align*}
where the functions $\mathfrak{I}_i$ are defined by, for $x\in \mathbb{Z}_+$:
\begin{align}\label{RecursionHarmonicFunctions}
\mathfrak{I}_i(x;\lambda)=\sum_{y=0}^{x-1}\frac{1}{\lambda(y)}\mathfrak{I}_{i-1}(y;\lambda), \ \mathfrak{I}_0\left(x;\lambda\right)\equiv 1.
\end{align}
\end{lem}

\begin{proof}
Direct computation by induction using multinearity of the determinant.
\end{proof}

\begin{rmk}
Observe that for $\lambda(\cdot)\equiv \mathbf{1}$ we have, for $x \in \mathbb{W}^N$:
\begin{align*}
\mathfrak{h}_N\left(x;\mathbf{1}\right)=\frac{\prod_{1\le i<j \le N}^{}(x_j-x_i)}{\prod_{j=1}^{N}(j-1)!}.
\end{align*}
This is the harmonic function associated to $N$ independent Poisson processes (i.e. with $\mathsf{L}=\nabla^+$) killed when they intersect, see \cite{NonCollidingWalks}, \cite{OConnellJPhys}, \cite{OConnellTams}, \cite{BorodinFerrari}.
\end{rmk}

\begin{rmk}
Lemma \ref{DeterminantRepresentationHarmonic} implies that the sequence of functions $\{\mathfrak{I}_{i}(\cdot;\lambda)\}_{i\ge 1}$ forms a (discrete) extended complete Chebyshev system on $\mathbb{Z}_+$, see \cite{Karlin}. On the real line and under certain assumptions, such systems have been classified and are characterized through a recurrence like (\ref{RecursionHarmonicFunctions}) (with integrals instead of sums), see \cite{Karlin}.
\end{rmk}

\begin{rmk}
It is possible to express the entries of the determinant representation for $\mathfrak{h}_N(x;\lambda)$ in terms of contour integrals as we shall see in Section 3. This is essential in order to perform the computation of the correlation kernel.
\end{rmk}

Now, we let $\Pi_N^{N+1}$ be the operator induced by the projection on the $y$-coordinates. More precisely, for a function $f$ on $\mathbb{W}^N$, the function $\Pi_N^{N+1}f$ on $\mathbb{W}^{N,N+1}$ is defined by $\left[\Pi_N^{N+1}f\right](y,x)=f(y)$. Then, we have:

\begin{prop}\label{PropInter1} For $t \ge 0$, we have the following equalities of positive kernels,
\begin{align}
\left[\Pi_{N}^{N+1}\mathcal{P}_t^{N}\right]((y,x),y')&=\left[\mathsf{U}_t^{N,N+1}\Pi_{N}^{N+1}\right]((y,x),y'), \ (y,x)\in \mathbb{W}^{N,N+1},y'\in \mathbb{W}^{N+1}.
\end{align}
\end{prop}
\begin{proof}
This computation is implicit in the proof of Lemma \ref{SubstochasticLemma}.
\end{proof}

Similarly we have:

\begin{prop}\label{PropInter2}For $t\ge 0$, we have the equalities of positive kernels,
 \begin{align}
 \left[\mathcal{P}_t^{N+1}\Lambda_{N}^{N+1}\right](x,(y',x'))&=\left[\Lambda_{N}^{N+1}\mathsf{U}_t^{N,N+1}\right](x,(y',x')), \ \ x\in \mathbb{W}^{N+1},(y',x') \in \mathbb{W}^{N,N+1}\label{intermediateintertwining1}.
 \end{align}
 \end{prop}
\begin{proof}
We take the sum $\sum_{\{y:(y,x)\in \mathbb{W}^{N,N+1} \}}^{}$ in the explicit form of the kernels and use multilinearity of the determinant. Then, the statement follows from the relations below 
 \begin{align*}
\sum_{y_i=x_i}^{x_{i+1}-1}\lambda(y_i)^{-1}\mathsf{C}_t(y,x')_{ij}&=\nabla^+_{x'_j} e^{t\mathsf{L}} \mathbf{1}_{[\![ 0,x_j']\!]}(x_{i+1})-\nabla^+_{x'_j} e^{t\mathsf{L} } \mathbf{1}_{[\![ 0,x_j']\!]}(x_{i}),\\
\sum_{y_i=x_i}^{x_{i+1}-1}\lambda(y_i)^{-1}\mathsf{D}_t(y,y')_{ij}&=-\lambda(y'_j)^{-1} e^{t\mathsf{L}} \mathbf{1}_{[\![ 0,y_j']\!]}(x_{i+1})+ \lambda(y'_j)^{-1}e^{t\mathsf{L}} \mathbf{1}_{[\![ 0,y_j']\!]}(x_{i})
 \end{align*}
 and simple row-column operations.
\end{proof}

Propositions \ref{PropInter1} and \ref{PropInter2} above readily imply the following two results:

\begin{prop}\label{KMIntertwining}
For $t\ge0$ we have:
\begin{align}
\left[\mathcal{P}_t^{N+1}\Lambda_N^{N+1}\right](x,y')=\left[\Lambda_N^{N+1}\mathcal{P}_t^N\right](x,y'), \ \ x\in \mathbb{W}^{N+1},y'\in \mathbb{W}^N.
\end{align}
\end{prop}

\begin{proof}
Combine Propositions \ref{PropInter1} and \ref{PropInter2}, noting that $\Lambda_N^{N+1}\Pi_N^{N+1}\equiv \Lambda_N^{N+1}$.
\end{proof}

\begin{prop}\label{HarmonicFunctionsProp}
The function $\mathfrak{h}_N(x)=\mathfrak{h}_N(x;\lambda)$ is a positive harmonic function for the semigroup $\left(\mathcal{P}_t^N;t\ge 0\right)$. Moreover, the function $\mathfrak{h}_{(N,N+1)}((y,x);\lambda)$ defined by $\mathfrak{h}_{(N,N+1)}((y,x);\lambda)=\mathfrak{h}_{N}(y;\lambda)$ is a positive harmonic function for the semigroup $\left(\mathsf{U}_t^{N,N+1};t\ge 0\right)$.
\end{prop}

\begin{proof}
Inductively apply Propositions \ref{KMIntertwining} and \ref{PropInter2} respectively.
\end{proof}

In order to proceed we require a general abstract definition. For a possibly sub-Markov semigroup $\left(\mathsf{P}(t);t \ge 0\right)$ having a strictly positive eigenfunction $\mathsf{h}$ with eigenvalue $e^{\mathsf{c}t}$ (i.e. $\mathsf{P}(t)\mathsf{h}=e^{\mathsf{c}t}\mathsf{h}$) we define its Doob $h$-transform by $\left(e^{-\mathsf{c}t}\mathsf{h}^{-1}\circ \mathsf{P}(t)\circ \mathsf{h};t\ge 0\right)$. We note that this is an honest Markovian semigroup. Thus, Proposition \ref{HarmonicFunctionsProp} allows us to correctly define the Doob h-transformed versions of the semigroups and kernels above:
\begin{align}
\mathfrak{L}_{N}^{N+1}(x,y)&=\frac{\mathfrak{h}_N(y)}{\mathfrak{h}_{N+1}(x)}\prod_{n=1}^{N}\frac{1}{\lambda(y_i)}\mathbf{1}(y\prec x), \ x\in \mathbb{W}^{N+1},y \in \mathbb{W}^N,\label{TransformedKernel}\\
\mathfrak{P}^N_t(x,y)&=\frac{\mathfrak{h}_N(y)}{\mathfrak{h}_N(x)}\det\left(e^{t\mathsf{L}}(x_i,y_j)\right)_{i,j=1}^N, \ t\ge 0, x,y \in \mathbb{W}^N,\label{TransformedSemigroup}\\
\mathfrak{U}_t^{N,N+1}\left[(y,x),(y',x')\right]&=\frac{\mathfrak{h}_{(N,N+1)}(y',x')}{\mathfrak{h}_{(N,N+1)}(y,x)}\mathsf{U}_t^{N,N+1}\left[(y,x),(y',x')\right], \ t\ge 0, (y,x),(y',x')\in \mathbb{W}^{N,N+1}.
\end{align}

Note that, by their very definition, all of these are now Markovian. Moreover, as we have done previously, we can also view $\mathfrak{L}_N^{N+1}$ as a Markov kernel from $\mathbb{W}^N$ to $\mathbb{W}^{N,N+1}$.  We observe that for the distinguished special case $\lambda(\cdot)\equiv 1$,  $\left(\mathfrak{P}_t^{N};t\ge 0\right)$ is the semigroup of the well-known Charlier process, see \cite{NonCollidingWalks}, \cite{OConnellJPhys}, \cite{OConnellTams}, \cite{BorodinFerrari} the discrete analogue of Dyson's Brownian motion \cite{DysonBrownian}. With all these preliminaries in place we have:

\begin{prop}\label{DoobTransformedIntertwinings} For $t\ge 0$, we have the intertwining relations between Markov semigroups:
 \begin{align}
 \left[\Pi_{N}^{N+1}\mathfrak{P}_t^{N}\right]((y,x),y')&=\left[\mathfrak{U}_t^{N,N+1}\Pi_{N}^{N+1}\right]((y,x),y'), \ (y,x)\in \mathbb{W}^{N,N+1},y'\in \mathbb{W}^{N+1},\\
 \left[\mathfrak{P}_t^{N+1}\mathfrak{L}_{N}^{N+1}\right](x,(y',x'))&=\left[\mathfrak{L}_{N}^{N+1}\mathfrak{U}_t^{N,N+1}\right](x,(y',x')), \ \ x\in \mathbb{W}^{N+1},(y',x') \in \mathbb{W}^{N,N+1},\label{MainIntertwining}\\
 \left[\mathfrak{P}_t^{N+1}\mathfrak{L}_N^{N+1}\right](x,y')&=\left[\mathfrak{L}_N^{N+1}\mathfrak{P}_t^N\right](x,y'), \ \ x\in \mathbb{W}^{N+1},y'\in \mathbb{W}^N.\label{KMIntertwiningTransformed}
 \end{align}
\end{prop}
\begin{proof}
These relations are straightforward consequences of Propositions \ref{PropInter1}, \ref{PropInter2} and \ref{KMIntertwining} respectively.
\end{proof}

Observe that, the $h$-transform by $\mathfrak{h}_{(N,N+1)}(y,x)=\mathfrak{h}_N(y)$ conditions the $Y$-particles to never collide and the process with semigroup $\left(\mathfrak{U}_t^{N,N+1};t\ge0\right)$ has infinite lifetime. Under this change of measure the evolution of the $Y$-particles is autonomous with semigroup $\left(\mathfrak{P}_t^N;t\ge0\right)$, while the $X$-particles evolve as $N+1$ independent chains with generator $\mathsf{L}$ interacting with the $Y$-particles through the same push-block dynamics of Definition \ref{TwoLevelDynamics}. We now arrive at the main result of this section.

\begin{thm}\label{MarkovFunctionTheorem}
Consider a Markov process $\left(\left(Y(t),X(t)\right);t\ge0\right)$ in $\mathbb{W}^{N,N+1}$ with semigroup $\left(\mathfrak{U}_t^{N,N+1};t \ge 0\right)$. Let $\mathfrak{M}^{N+1}$ be a probability measure on $\mathbb{W}^{N+1}$. Assume $\left(\left(Y(t),X(t)\right);t\ge0\right)$ is initialized according to the probability measure with density $\mathfrak{M}^{N+1}(x)\mathfrak{L}_N^{N+1}(x,y)$ on $\mathbb{W}^{N,N+1}$. Then, the projection on the $X$-particles is distributed as a Markov process with semigroup $\left(\mathfrak{P}_t^{N+1};t \ge 0\right)$ and initial condition $\mathfrak{M}^{N+1}$. Moreover, for any fixed time $T\ge 0$, the conditional distribution of $\left(X(T),Y(T)\right)$ given $X(T)$ satisfies:
\begin{align}\label{ConditionalLaw}
\mathsf{Law}\left[\left(X(T),Y(T)\right)\big|X(T)\right]=\mathfrak{L}_N^{N+1}\left(X(T),\cdot\right).
\end{align}
\end{thm}
\begin{proof}
Let $\mathsf{S}$ be the operator induced by the projection on the $x$-coordinates:
\begin{align*}
\left[\mathsf{S}f\right](y,x)=\left[f\circ s\right](y,x), \ s(y,x)=x,
\end{align*}
(we do not indicate dependence on $N$). Observe that,
\begin{align*}
\mathfrak{L}_N^{N+1}\mathsf{S}=\textnormal{Id}, \ \textnormal{on} \ \mathbb{W}^{N+1}.
\end{align*}
Then, the first statement of the theorem, by virtue of the intertwining relation (\ref{MainIntertwining}), is an application of the theory of Markov functions due to Rogers and Pitman, see Theorem 2 in \cite{RogersPitman} (applied to the function $s$ above). Finally, for the conditional law statement (\ref{ConditionalLaw}) see Remark (ii) following Theorem 2 of \cite{RogersPitman}.
\end{proof}

\begin{rmk}\label{LevelInhomogeneousRemark2}
Assume we are in the setting of Remark \ref{LevelInhomogeneousRemark1}. Let $\mathsf{L}^1$ and $\mathsf{L}^2$ be two pure-birth chain generators:
\begin{align*}
\mathsf{L}^1_{x}=\left(\lambda(x)+\beta_1\right)\nabla_x^+, \ \mathsf{L}^2_{x}=\left(\lambda(x)+\beta_2\right)\nabla_x^+.
\end{align*}
We observe that the strictly positive eigenfunction $\mathsf{h}_{\beta_1}^{\beta_2}$ of $\mathsf{L}^2$ (with eigenvalue $\beta_1-\beta_2$) defined by:
\begin{align*}
\mathsf{h}_{\beta_1}^{\beta_2}(x)=p_x(\beta_2-\beta_1;\lambda(\cdot)+\alpha_2)=\prod_{l=0}^{x-1}\frac{\lambda(l)+\beta_1}{\lambda(l)+\beta_2}
\end{align*}
Doob h-transforms $\mathsf{L}^2$ to $\mathsf{L}^1$:
\begin{align*}
\left(\mathsf{h}_{\beta_1}^{\beta_2}\right)^{-1}\circ \mathsf{L}^2\circ \mathsf{h}_{\beta_1}^{\beta_2}+\left(\beta_2-\beta_1\right)\mathsf{I}=\mathsf{L}^1.
\end{align*}
We define, for $n\ge 1$:
\begin{align*}
\mathsf{\Lambda}_n^{(n+1,\alpha_{n+1})}(x,y)&=\prod_{i=1}^{n}\frac{1}{\lambda(y_i)+\alpha_{n+1}}\mathbf{1}(y\prec x),\\
\mathsf{P}_t^{(n,\alpha_n)}(x,y)&=\det \left(e^{t\left(\lambda(\cdot)\nabla_\cdot^++\alpha_n\right)}(x_i,y_j)\right)_{i,j=1}^n,\\
\mathsf{h}_{n+1}^{(\alpha_1,\dots,\alpha_{n+1})}(x_1,\dots,x_{n+1})&=\left[\mathsf{\Lambda}_n^{(n+1,\alpha_{n+1})}\prod_{i=1}^{n}\mathsf{h}_{\alpha_n}^{\alpha_{n+1}}(\cdot)\mathsf{h}_n^{(\alpha_1,\dots,\alpha_n)}(\cdot)\right](x_1,\dots,x_{n+1}), \ \mathsf{h}^{\alpha_1}_1\equiv 1.
\end{align*}
By an inductive argument, making use of Proposition \ref{KMIntertwining}, we can show that $\mathsf{h}_{n}^{(\alpha_1,\dots,\alpha_n)}$ is a strictly positive eigenfunction of $\left(\mathsf{P}_t^{(n,\alpha_n)};t\ge 0\right)$. Thus, we can consider the Doob h-transformed versions $\mathsf{P}_t^{(n,\alpha_n),\mathsf{h}_n^{(\alpha_1,\dots,\alpha_n)}}$ and $\mathsf{\Lambda}_n^{(n+1,\alpha_{n+1}),\mathsf{h}_n^{(\alpha_1,\dots,\alpha_n)}}$. Then, all of the results above have natural extensions involving these quantities (whose precise statements we omit) to the level inhomogeneous setting.
\end{rmk}

\subsection{Consistent multilevel dynamics}
We have the following multilevel extension of the results of the preceding subsection.
\begin{prop}\label{ConsistentMultilevelProp}
 Let $\left(\mathfrak{P}_{t}^{k};t \ge 0\right)$ and $\mathfrak{L}^{k}_{k-1}$ denote the semigroups and Markov kernels defined in (\ref{TransformedSemigroup}) and (\ref{TransformedKernel}) above and let $\mathfrak{M}^N(\cdot)$ be a probability measure on $\mathbb{W}^N$. Define the following Gibbs probability measure $\mathsf{M}_N$ on $\mathsf{GT}_N$ with density:
\begin{align}\label{GibbsTypeA}
\mathsf{M}_N(x^1,\dots,x^N)=\mathfrak{M}^N(x^N)\mathfrak{L}^{N}_{N-1}\left(x^N,x^{N-1}\right)\cdots \mathfrak{L}^{2}_{1}\left(x^2,x^{1}\right).
\end{align}
Consider the process $\left( \mathsf{X}_N\left(t;\mathsf{M}_N\right);t\ge0\right)=\left(\left( \mathsf{X}^1\left(t\right),\mathsf{X}^2\left(t\right),\dots,\mathsf{X}^N\left(t\right)\right);t\ge0\right)$ in Definition \ref{BFDynamics}. Then, for $1\le k \le N$ the projection  on the $k$-th level $\left(\mathsf{X}^{k}(t);t \ge 0\right)$ is distributed as a Markov process evolving according to $\left(\mathfrak{P}_t^{k};t \ge 0\right)$. Moreover, for any fixed $T\ge 0$, the law of $\left(\mathsf{X}^1(T),\dots, \mathsf{X}^{N}(T)\right)$ is given by the evolved Gibbs measure on $\mathsf{GT}_N$ :
\begin{align}\label{EvolvedGibbsTypeA}
\mathsf{Law}\left[\mathsf{X}_N\left(T;\mathsf{M}_N\right)\right]=\left[\mathfrak{M}^N\mathfrak{P}^{N}_T\right](\cdot)\mathfrak{L}^{N}_{N-1}\left(\cdot,\cdot\right)\cdots \mathfrak{L}^{2}_{1}\left(\cdot,\cdot\right).
\end{align}
\end{prop}
\begin{proof}
The proof is by induction. For $N=2$, this is Theorem \ref{MarkovFunctionTheorem}. Assume the result is true for $N-1$ and we prove it for $N$. We first observe that the induced measure on $\mathsf{GT}_{N-1}$
\begin{align*}
\left(\pi_{N-1}^N\right)_*\mathsf{M}_N\left(x^1,\dots,x^{N-1}\right)=\left[\mathfrak{M}^N\mathfrak{L}^{N}_{N-1}\right]\left(x^{N-1}\right)\mathfrak{L}^{N-1}_{N-2}\left(x^{N-1},x^{N-2}\right)\cdots \mathfrak{L}^{2}_{1}\left(x^2,x^{1}\right)
\end{align*}
is again Gibbs. Then, from the induction hypothesis $\left(\mathsf{X}^{N-1}(t);t \ge 0\right)$ is a Markov process with semigroup $\left(\mathfrak{P}^{N-1}_t;t \ge 0\right)$. Moreover, the joint dynamics of $\left(\mathsf{X}^{N-1}(t),\mathsf{X}^N(t);t \ge 0\right)$ are those considered in Theorem \ref{MarkovFunctionTheorem} (with semigroup $\mathfrak{U}_t^{N-1,N}$) and thus by the aforementioned result, we obtain that $\left(\mathsf{X}^N(t);t \ge 0\right)$ is distributed as a Markov process with semigroup $\left(\mathfrak{P}^{N}_t;t \ge 0\right)$. Furthermore by the same theorem we have that, for fixed $T\ge0$, the conditional law of $\mathsf{X}^{N-1}(T)$ given $\mathsf{X}^{N}(T)$ is $\mathfrak{L}^{N}_{N-1}\left(\mathsf{X}^N(T),\cdot\right)$. Hence, since the distribution of $\mathsf{X}^N(T)$ has density $\left[\mathfrak{M}^N\mathfrak{P}^{N}_T\right](\cdot)$, we get by the induction hypothesis, that the fixed time $T\ge0$, distribution of $\left(\mathsf{X}^1(T),\dots, \mathsf{X}^{N}(T)\right)$ is given by (\ref{EvolvedGibbsTypeA}) as desired.
\end{proof}

We observe that the densely packed initial condition $\mathsf{M}_N^{\mathsf{dp}}$ is clearly Gibbs. We close this subsection with a couple of remarks on generalizations of this result.

\begin{rmk}
It is also possible, by a simple extension of the argument above, to consider the distribution of $\left(\mathsf{X}^1(T_1),\mathsf{X}^2(T_2),\dots, \mathsf{X}^{N}(T_N)\right)$ at distinct times $(T_1,\dots,T_N)$ satisfying $T_N\le T_{N-1}\le\cdots \le T_1$. This corresponds to space-like distributions in the language of growth models, see  \cite{BorodinFerrariPushASEP}, \cite{BorodinFerrari}, \cite{BorodinFerrariPrahoferSasamoto}.
\end{rmk}

\begin{rmk}\label{LevelInhomogeneousRemark3}
In the setting of the level inhomogeneous model described in Remark \ref{LevelInhomogeneousRemark1} (with the notations of Remark \ref{LevelInhomogeneousRemark2}) the statement of the corresponding proposition (and its proof) is completely analogous with $\mathfrak{P}_t^{k}$ replaced by $\mathsf{P}_t^{(k,\alpha_k),\mathsf{h}_k^{(\alpha_1,\dots,\alpha_k)}}$ and $\mathfrak{L}_{k-1}^k$ replaced by $\mathsf{\Lambda}_{k-1}^{(k,\alpha_k),\mathsf{h}_k^{(\alpha_1,\dots,\alpha_k)}}$.
\end{rmk}

\subsection{Inhomogeneous Gelfand-Tsetlin graph and Plancherel measure}
This subsection is independent to the rest of the paper and can be skipped. However, it provides some further insight into the constructions of the present work and how they fit into a wider framework. We begin with some notation. Let
\begin{align*}
\tilde{\mathbb{W}}^N=\{(x_1,\dots,x_N)\in \mathbb{Z}^N:x_1<\dots<x_N\}
\end{align*}
denote the discrete chamber without the non-negativity restriction. The definitions of interlacing in this setting and of $\tilde{\mathbb{W}}^{N,N+1}$ are also completely analogous (we simply drop non-negativity).
\begin{defn}
We consider a graded graph $\mathsf{\Gamma}=\mathsf{\Gamma}_{\lambda}$ with vertex set $\uplus_{N\ge 1}\tilde{\mathbb{W}}^N$. Two vertices $x\in \tilde{\mathbb{W}}^{N+1}$ and $y\in \tilde{\mathbb{W}}^N$ are connected by an edge if and only if they interlace. For all $N\ge 1$ we assign a weight/multiplicity, denoted by $\mathsf{mult}_{\lambda}(y,x)$, to each edge $(y,x) \in \tilde{\mathbb{W}}^{N,N+1}$, and more generally to all pairs $(y,x)\in \tilde{\mathbb{W}}^N\times\tilde{\mathbb{W}}^{N+1}$:
\begin{align*}
\mathsf{mult}_{\lambda}(y,x)=\prod_{i=1}^{N}\frac{1}{\lambda(y_i)}\mathbf{1}(y\prec x).
\end{align*}
\end{defn}
The distinguished case $\mathsf{\Gamma}_1$ with $\lambda(\cdot)\equiv 1$ is the Gelfand-Tsetlin graph\footnote{The Gelfand-Tsetlin graph vertex set is commonly defined in terms of signatures $\mathsf{Sign}_N=\{\nu=(\nu_1,\dots,\nu_N) \in \mathbb{Z}^N:\nu_1\ge \nu_2\ge \dots \ge \nu_N \}$ for which there is a corresponding notion of interlacing which then gives the edge set. The two definitions are equivalent since there is a natural bijection between $\mathsf{Sign}_N$ and $\mathbb{W}^N$ under which interlacing in terms of signatures becomes interlacing in terms of elements of the $\mathbb{W}^N$'s.}, see \cite{VershikKerov}, \cite{BorodinOlshanskiBoundary}. This describes the branching of irreducible representations of the chain of unitary groups, see \cite{VershikKerov}, \cite{BorodinOlshanskiBoundary}. We propose to call the more general case $\mathsf{\Gamma}_{\lambda}$ defined above the inhomogeneous Gelfand-Tsetlin graph (in fact it is a family of graphs, one for each function $\lambda$).

We now define the dimension $\mathsf{dim}^{\lambda}_N(x)$ of a vertex $x \in \tilde{\mathbb{W}}^N$, inductively by:
\begin{align*}
\mathsf{dim}^{\lambda}_{k+1}(x)=\sum_{y\prec x}\mathsf{dim}^{\lambda}_{k}(y)\mathsf{mult}_{\lambda}(y,x), \ x\in \tilde{\mathbb{W}}^{k+1}, y\in \tilde{\mathbb{W}}^k\ \textnormal{ with } \mathsf{dim}^{\lambda}_1(z)\equiv 1, z \in \mathbb{Z}.
\end{align*}
We can associate a family of Markov kernels $\{\Lambda_{N+1\to N}\}_{N\ge 1}$ from $\tilde{\mathbb{W}}^{N+1}$ to $\tilde{\mathbb{W}}^N$ to the graph $\mathsf{\Gamma}_{\lambda}$ given by:
\begin{align*}
\Lambda_{N+1\to N}(x,y)=\frac{\mathsf{dim}^{\lambda}_{N}(y)\mathsf{mult}_{\lambda}(y,x)}{\mathsf{dim}^{\lambda}_{N+1}(x)},\  x\in \tilde{\mathbb{W}}^{N+1}, y\in \tilde{\mathbb{W}}^N.
\end{align*}
Observe that, by the very definitions, when restricting to the positive chambers $\mathbb{W}^N$ (namely considering the subgraph $\mathsf{\Gamma}_{\lambda}^+=\uplus_{N\ge 1} \mathbb{W}^N$) we have:
\begin{align}\label{EqualityofKernels}
\mathsf{dim}^{\lambda}_{k+1}(x)&=\mathfrak{h}_N(x;\lambda), \ x \in \mathbb{W}^N,\nonumber \\ \Lambda_{N+1\to N}(x,y)&=\mathfrak{L}_N^{N+1}(x,y), \ x\in \mathbb{W}^{N+1},y \in \mathbb{W}^N.
\end{align}

We say that a sequence of probability $\{ \mu_N\}_{N\ge 1}$ on $\{\tilde{\mathbb{W}}^N \}_{N\ge 1}$ is consistent if:
\begin{align*}
\mu_{N+1}\Lambda_{N+1\to N}=\mu_N, \ \forall N\ge 1.
\end{align*}
The extremal points of the convex set of consistent probability measures form the boundary of the graph $\mathsf{\Gamma}_{\lambda}$. In the homogeneous case $\lambda(\cdot)\equiv 1$, the boundary of the Gelfand-Tsetlin graph $\mathsf{\Gamma}_1$ has been determined explicitly and is in bijection (see \cite{VershikKerov}, \cite{BorodinOlshanskiBoundary}, \cite{PetrovBoundary} for more details and precise statements) with the infinite dimensional space $\Omega$:
\begin{align*}
&\Omega=(\alpha^+,\alpha^-,\beta^+,\beta^-,\delta^+,\delta^-)\in \mathbb{R}^{4\infty+2},\\
&\alpha^{\pm}=\left(\alpha_1^{\pm}\ge \alpha_2^{\pm}\ge \dots \ge 0\right)\in \mathbb{R}^{\infty}, \ \beta^{\pm}=\left(\beta_1^{\pm}\ge\beta_2^{\pm}\ge\dots\ge 0\right)\in\mathbb{R}^{\infty}, \ \delta^{\pm} \in \mathbb{R},\\
&\sum_i^{\infty}\left(\alpha_i^{\pm}+\beta_i^{\pm}\right)\le \delta^{\pm}, \ \beta_1^++\beta_1^-\le 1,
\end{align*}
and we also write:
\begin{align*}
\gamma^{\pm}=\delta^{\pm}-\sum_{i=1}^{\infty}\left(\alpha_i^{\pm}+\beta_i^{\pm}\right)\ge 0.
\end{align*}
The extremal consistent sequence of probability measures $\big\{\mathcal{M}_{\gamma^+}^N \big\}_{N\ge 1}$ corresponding to $\gamma^+\ge0$ with all the other parameters on $\Omega$ identically equal to zero is called the Plancherel measure\footnote{More generally the measure where both parameters $(\gamma^+,\gamma^-)$ can be positive is also called Plancherel.} for the infinite dimensional unitary group, see \cite{BorodinKuanPlancherel}. The connection to the present paper is through the following, see \cite{BorodinFerrari}
\begin{align*}
\mathcal{M}_{\gamma^+}^N(\cdot)= \mathfrak{P}_{\gamma^+}^N\left((0,1,\dots,N-1),\cdot\right), \forall N\ge 1,
\end{align*}
where the right hand side is defined for $\lambda(\cdot)\equiv 1$. Now, due to observation (\ref{EqualityofKernels}) and the fact that $\mathfrak{P}_{\gamma^+}^N\left((0,1,\dots,N-1),\cdot\right)$ is supported on $\mathbb{W}^N$ the following is an immediate consequence of the intertwining relation (\ref{KMIntertwiningTransformed}) from Proposition \ref{DoobTransformedIntertwinings}:

\begin{prop}
Let the function $\lambda$ be fixed satisfying ($\mathsf{UB}$). Consider the graph $\mathsf{\Gamma}_{\lambda}$ and for all $N\ge 1$ the semigroups $\left(\mathfrak{P}_t^{N};t\ge 0\right)$ associated to the function $\lambda$. Then, for each $\gamma^+\ge 0$ the sequence of probability measures $\big\{\mathfrak{P}_{\gamma^+}^N\left((0,1,\dots,N-1),\cdot\right)\big\}_{N\ge 1}$ is consistent for $\mathsf{\Gamma}_{\lambda}$.
\end{prop}

Thus, the sequence $\big\{\mathfrak{P}_{\gamma^+}^N\left((0,1,\dots,N-1),\cdot\right)\big\}_{N\ge 1}$ can be viewed as the analogue of the Plancherel measure for the more general graphs $\mathsf{\Gamma}_{\lambda}$. It would be interesting to understand whether this sequence is actually extremal for $\mathsf{\Gamma}_{\lambda}$ for general $\lambda$. A more ambitious question would be whether there exists a complete classification of extremal consistent measures for $\Gamma_{\lambda}$, in analogy to the case of the Gelfand-Tsetlin graph $\mathsf{\Gamma}_1$.

\begin{rmk}
Analogous constructions exist for the level-inhomogeneous generalization of the Gelfand-Tsetlin graph c.f. Remarks \ref{LevelInhomogeneousRemark1}, \ref{LevelInhomogeneousRemark2}, \ref{LevelInhomogeneousRemark3}.
\end{rmk}

\section{Determinantal structure and computation of the kernel}

\subsection{Eynard-Mehta Theorem and determinantal correlations}

We will make use of one of the many variants of the famous Eynard-Mehta Theorem \cite{EynardMehta}, and in particular a generalization to measures on interlacing particle systems, see \cite{BorodinRains}, \cite{BorodinFerrariPrahoferSasamoto}. More precisely, we will use Lemma 3.4 of \cite{BorodinFerrariPrahoferSasamoto}. For the convenience of the reader and to set up some notation we reproduce it here:

\begin{prop}\label{EynardMehta}
Assume we have a (possibly signed) measure on $\{x_i^n, i=1,\dots, N, i=1,\dots,n\}$ given in the form:
\begin{align}
\frac{1}{Z_N}\prod_{n=1}^{N-1}\det\left[\phi_n(x_i^n,x_j^{n+1})\right]_{i,j=1}^{n+1}\det\left[\Psi_{N-i}^N\left(x_j^N\right)\right]_{i,j=1}^N,
\end{align}
where $x_{n+1}^{n}$ are some "virtual" variables, which we also denote by $\mathsf{virt}$, and $Z_N$ is a non-zero normalization constant. Then, the correlation functions are determinantal. To write down the kernel we need some notation. Define,
\begin{align*}
\phi^{(n_1,n_2)}(x,y)=\begin{cases}
\left(\phi_{n_1}* \cdots *\phi_{n_2}\right)(x,y), & n_1<n_2,\\
0, & n_1\ge n_2,
\end{cases}
\end{align*}
where $(a*b)(x,y)=\sum_{z \in \mathbb{Z}}^{}a(x,z)b(z,y)$. Also, define for $1\le n <N$:
\begin{align*}
\Psi_{n-j}^n(y)=\left(\phi^{(n,N)}*\Psi_{N-j}^N\right)(y), \ j=1,2,\dots, N.
\end{align*}
Set $\phi(x_1^0,x)=1$. Then, the functions
\begin{align*}
\big\{ (\phi_0*\phi^{(1,n)})(x_1^0,x),\dots,(\phi_{n-2}*\phi^{(n-1,n)})(x_{n-1}^{n-2},x),\phi_{n-1}(x_{n}^{n-1},x)\big\}
\end{align*}
are linearly independent and generate the $n$-dimensional space $V_n$. For each $1\le n \le N$ we define a set of functions $\{\Phi_j^n(x), j=0,\dots,n-1 \}$ determined by the following two properties:
\begin{itemize}
\item The functions $\{\Phi_j^n(x), j=0,\dots,n-1 \}$ span $V_n$.
\item For $1\le i,j \le n-1$ we have:
\begin{align*}
\sum_{x}^{}\Psi_i^n(x)\Phi_j^n(x)=\mathbf{1}(i=j).
\end{align*}
\end{itemize}
Finally, assume that $\phi_n(x_{n+1}^n,x)=c_n\Phi_0^{(n+1)}(x)$ for some $c_n\neq 0$, $n=1,\dots,N-1$. Then, the kernel takes the simple form:
\begin{align}
K(n_1,x_1;n_2,x_2)=-\phi^{(n_1,n_2)}(x_1,x_2)+\sum_{k=1}^{n_2}\Psi^{n_1}_{n_1-k}(x_1)\Phi_{n_2-k}^{n_2}(x_2).
\end{align}
\end{prop}

From Proposition \ref{ConsistentMultilevelProp} we get that $\mathsf{Law}\left[\mathsf{X}_N(t;\mathsf{M}_{N}^{\mathsf{dp}})\right]$ for fixed time $t\ge 0$ is given by, where we use the notation $\triangle_N=(0,1,\dots,N-1)$:
\begin{align*}
\frac{\mathfrak{h}_N(x^N)}{\mathfrak{h}_N(\triangle_N)}\det\left(e^{t \mathsf{L}}(i-1,x^N_j)\right)_{i,j=1}^N \frac{\mathfrak{h}_{N-1}\left(x^{N-1}\right)}{\mathfrak{h}_{N}\left(x^{N}\right)}\prod_{j=1}^{N-1}\frac{1}{\lambda\left(x_j^{N-1}\right)}\mathbf{1}\left(x^{N-1}\prec x^N\right) \cdots \frac{\mathfrak{h}_1(x^1)}{\mathfrak{h}_2(x^2)}\frac{1}{\lambda(x^1)}\mathbf{1}\left(x^1\prec x^2\right).
\end{align*}
Using the spectral expansion (\ref{SpectralExpansion}) of $e^{t\mathsf{L}}(x,y)$ and row operations (recall that $p_x(w)$ is a polynomial of degree $x$ in $w$) we can rewrite the display above as follows, for a (different) non-zero constant $\mathsf{Z}_N$:
\begin{align*}
\frac{1}{\mathsf{Z}_N}\det\left(\Psi_{N-i}^N(x_j^N)\right)_{i,j=1}^N \prod_{j=1}^{N-1}\frac{1}{\lambda\left(x_j^{N-1}\right)}\mathbf{1}\left(x^{N-1}\prec x^N\right) \cdots \prod_{j=1}^{2}\frac{1}{\lambda\left(x_j^{2}\right)}\mathbf{1}\left(x^{2}\prec x^3\right)\frac{1}{\lambda(x^1)}\mathbf{1}\left(x^1\prec x^2\right),
\end{align*}
where the functions $\{ \Psi_{N-i}^N(\cdot)\}_{i=1}^N$ (we suppress dependence on the time variable $t$ since it is fixed) are given by:
\begin{align}
\Psi^N_{N-i}(x)=-\frac{1}{\lambda(x)}\frac{1}{2\pi \i} \oint_{\mathsf{C}_{\lambda}}\psi_x(w)w^{N-i}e^{-tw}dw.
\end{align}
Moreover, we note that it is possible (see for example \cite{Warren}) to write the indicator function for interlacing as a determinant, for $y \in \mathbb{W}^{n-1}, x \in \mathbb{W}^{n}$:
\begin{align*}
\mathbf{1}(y \prec x)=\det \left(\mathsf{f}_{ij}\right)_{i,j=1}^n, \ \textnormal{where} \  \mathsf{f}_{ij}=\begin{cases}
-\mathbf{1}\left(x_i>y_j\right), \ & j\le n-1,\\
1, \ & j=n.
\end{cases}
\end{align*}
Thus, we can write the measure above in a form that is within the scope of Proposition \ref{EynardMehta}:
\begin{align*}
\frac{1}{\mathsf{Z}_N}\det\left(\Psi_{N-i}^N(x_j^N)\right)_{i,j=1}^N \det\left[\phi_{N-1}\left(x_i^{N-1},x_j^N\right)\right]_{i,j=1}^N \cdots\det\left[\phi_{2}\left(x_i^{2},x_j^3\right)\right]_{i,j=1}^3 \det\left[\phi_{1}\left(x_i^{1},x_j^2\right)\right]_{i,j=1}^2
\end{align*}
with,
\begin{align}
\phi_k(y,x)\equiv\phi(y,x)= \begin{cases}
-\frac{1}{\lambda(y)}\mathbf{1}(x>y), &y \in \mathbb{Z}_+,\\
1, & y= \mathsf{virt},
\end{cases} \ \ 1\le k \le N-1.
\end{align}
In particular this implies determinantal correlations. The explicit computation of the kernel is performed in the next section.

\begin{rmk}
A completely analogous computation as the one above gives that starting from any deterministic initial condition $\mathfrak{M}^N(\cdot)=\delta_{(z_1,\dots,z_N)}(\cdot)$ for the top level the evolved Gibbs measure on $\mathsf{GT}_N$ has determinantal correlations. A possible choice (this is clearly not unique, as we can use linear combinations) of the functions $\Psi$ is as follows (the functions $\phi$ are as before):
\begin{align*}
\Psi_{N-i}^{N}(x)=e^{t\mathsf{L}}\left(z_i,x\right), \ i=1,\dots,N.
\end{align*}
The explicit computation of the correlation kernel is an interesting open problem.
\end{rmk}

\begin{rmk}\label{LevelInhomogeneousRemark4}
In the level inhomogeneous setting, analogous computations, making use of Remark \ref{LevelInhomogeneousRemark3} show the existence of determinantal correlations for Gibbs measures with deterministic initial conditions for the top level $\mathfrak{M}^N(\cdot)=\delta_{(z_1,\dots,z_N)}(\cdot)$. A possible choice of the functions $\Psi$ and $\phi_k$ is as follows:
\begin{align*}
\Psi_{N-i}^{N}(x)&=e^{t\left(\lambda(\cdot)\nabla^+_\cdot+\alpha_N\right)}\left(z_i,x\right),\ 1\le i \le N, \\
\phi_k(y,x)&=-\frac{1}{\lambda(y)+\alpha_{k+1}}\mathsf{h}_{\alpha_k}^{\alpha_{k+1}}(y)\mathbf{1}(y<x), \ 1 \le k \le N-1.
\end{align*}
\end{rmk}

\subsection{Computation of the correlation kernel}
Our aim now is to solve the biorthogonalization problem given in Proposition \ref{EynardMehta} and obtain concise contour integral expressions for the families of functions appearing therein. This is achieved in the following sequence of lemmas. Firstly in order to ease notation, since we are in the level homogeneous case, we define:
\begin{align*}
\phi^{(n_2-n_1)}(z_1,z_2)=\phi^{(n_1,n_2)}(z_1,z_2).
\end{align*}

\begin{lem}\label{AuxLemma1} For $1\le n \le N$, we have:
\begin{align}
\Psi^n_{n-j}(x)=-\frac{1}{\lambda(x)}\frac{1}{2\pi \i} \oint_{\mathsf{C}_{\lambda}}\psi_x(w)w^{n-j}e^{-tw}dw, \ j=1,2,\dots, N.
\end{align}
\end{lem}

\begin{lem}\label{AuxLemma2}
For $1\le k\le N$, we have:
\begin{align*}
\phi^{(k)}(y,x)=-\frac{1}{\lambda(y)} \frac{1}{2\pi \i} \oint_{\mathsf{C_\lambda}}\psi_y(w)\frac{p_x(w)}{w^k}dw, \ x,y \in \mathbb{Z}_+.
\end{align*}
\end{lem}

\begin{lem}\label{AuxLemma3}
For $1\le k \le N$, we have:
\begin{align}
\phi^{(k)}(\mathsf{virt},x)=\frac{1}{2\pi \i} \oint_{\mathsf{C}_0}\frac{p_x(w)}{w^k}dw, \ x\in \mathbb{Z}_+.
\end{align}
\end{lem}

We define a family of functions $\Phi_{\cdot}^{\cdot}(\cdot)$ on $\mathbb{Z}_+$, for $1\le n \le N$ (again we suppress dependence on the variable $t$ since it is fixed):
\begin{align}\label{PhiFunctionDefinition}
\Phi^n_{n-j}(x)=\frac{1}{2\pi \i} \oint_{\mathsf{C}_0}\frac{p_x(u)}{e^{-tu}u^{n-j+1}}du, j=1,2,\dots,n.
\end{align}

\begin{lem}\label{AuxLemma4} The functions $\Phi$ are biorthogonal to the $\Psi$'s. More precisely, for any $1\le n \le N$
\begin{align*}
\sum_{x\ge 0}^{}\Psi_i^n(x)\Phi_j^n(x)=\mathbf{1}(i=j)
\end{align*}
for $1\le i,j \le n-1$.
\end{lem}

\begin{lem}\label{AuxLemma5}
The functions $\{ \Phi_j^n(\cdot); j=0,\dots,n-1\}$ span the space (see Proposition \ref{EynardMehta}):
\begin{align}
V_n=span\big\{\phi^{(1)}(\mathsf{virt},\cdot),\phi^{(2)}(\mathsf{virt},\cdot),\dots,\phi^{(n)}(\mathsf{virt},\cdot) \big\},
\end{align}
for $1\le n \le N$.
\end{lem}

Finally, it is clear that we have the following.

\begin{lem}
For $n=1,\dots,N-1$:
\begin{align*}
\phi(\mathsf{virt},\cdot)=\Phi_0^{(n+1)}(\cdot)\equiv 1.
\end{align*}
\end{lem}

Assuming the auxiliary results above we first prove Theorem \ref{MainTheorem}. The proofs of these results are given afterwards.

\begin{proof}[Proof of Theorem \ref{MainTheorem}]
Making use of Proposition \ref{EynardMehta} (by virtue of the preceding auxiliary lemmas) we get that:
\begin{align*}
\mathsf{K}_t(n_1,x_1;n_2,x_2)=-\phi^{(n_1,n_2)}(x_1,x_2)+\sum_{k=1}^{n_2}\Psi^{n_1}_{n_1-k}(x_1)\Phi_{n_2-k}^{n_2}(x_2).
\end{align*}
The term $\phi^{(n_2-n_1)}(x_1,x_2)$ is given by, from Lemma \ref{AuxLemma2}:
\begin{align*}
\phi^{(n_2-n_1)}(x_1,x_2)=-\frac{1}{\lambda(x_1)} \frac{1}{2\pi \i} \oint_{\mathsf{C_\lambda}}\psi_{x_1}(w)\frac{p_{x_2}(w)}{w^{n_2-n_1}}dw\mathbf{1}(n_2>n_1).
\end{align*}
It then remains to simplify the sum:
\begin{align*}
\sum_{k=1}^{n_2}\Psi^{n_1}_{n_1-k}(x_1)\Phi_{n_2-k}^{n_2}(x_2)&=-\frac{1}{\lambda(x_1)}\sum_{k=1}^{n_2}\frac{1}{(2\pi \i)^2} \oint_{\mathsf{C}_\lambda}\psi_{x_1}(w)w^{n_1-k}e^{-tw}dw\oint_{\mathsf{C}_0}\frac{p_{x_2}(u)}{e^{-tu}u^{n_2-k+1}}du\\
&=-\frac{1}{\lambda(x_1)}\frac{1}{(2\pi \i)^2}\oint_{\mathsf{C}_\lambda}dw\oint_{\mathsf{C}_0}du\psi_{x_1}(w)p_{x_2}(u)e^{-t(w-u)}\sum_{k=1}^{n_2}\frac{w^{n_1-k}}{u^{n_2-k+1}}\\
&=-\frac{1}{\lambda(x_1)}\frac{1}{(2\pi \i)^2}\oint_{\mathsf{C}_\lambda}dw\oint_{\mathsf{C}_0}du\psi_{x_1}(w)p_{x_2}(u)e^{-t(w-u)}\frac{w^{n_1-n_2}}{u^{n_2}}\frac{w^{n_2}-u^{n_2}}{w-u}.
\end{align*}
\end{proof}

\begin{proof}[Proof of Lemma \ref{AuxLemma1}]
It suffices to show that:
\begin{align*}
\left(\phi*\Psi_{N-j}^N\right)(y)=\Psi_{N-1-j}^{N-1}(y), \ y \in \mathbb{Z}_+.
\end{align*}
It is in fact equivalent to prove that:
\begin{align}\label{AuxLem1Eq1}
\sum_{x\ge 0}\left(-\frac{1}{\lambda(y)}\right)\mathbf{1}\left(y< x\right)\left(-\frac{1}{\lambda(x)}\right)\frac{1}{2\pi \i} \oint_{\mathsf{C}_{\lambda}}\psi_x(w)w^{i-1}e^{-tw}dw=-\frac{1}{\lambda(y)}\frac{1}{2\pi \i} \oint_{\mathsf{C}_{\lambda}}\psi_y(w)w^{i-2}e^{-tw}dw.
\end{align}
For any $R>1$, we consider a counter clockwise contour $\mathsf{C}_{\lambda}^{\ge R}$ that contains $0$ and $\{\lambda(x)\}_{x\ge0}$ and for which the following uniform bound holds:
\begin{align}
\sup_{w \in \mathsf{C}_{\lambda}^{\ge R}}\sup_{x \in \mathbb{Z}_+}\bigg|\frac{\lambda(x)}{\lambda(x)-w}\bigg|\le \frac{1}{R}.
\end{align}
Such a contour exists because of assumption $(\mathsf{UB})$; we can simply take a very large circle.

Clearly, the left hand side of display (\ref{AuxLem1Eq1}) is equal to (since we can deform the contour $\mathsf{C}_{\lambda}$ to $\mathsf{C}_{\lambda}^{\ge R}$ without crossing any poles):
\begin{align*}
-\frac{1}{\lambda(y)}\frac{1}{2\pi \i}\sum_{x>y}\left(-\frac{1}{\lambda(x)}\right) \oint_{\mathsf{C}_{\lambda}^{\ge R}}\psi_x(w)w^{i-1}e^{-tw}dw.
\end{align*}
We now claim that uniformly for $w \in \mathsf{C}_{\lambda}^{\ge R}$:
\begin{align}\label{KeyEquality}
-\sum_{x>y}^{}\frac{\psi_x(w)}{\lambda(x)}=\frac{\psi_y(w)}{w}.
\end{align}
Assuming this, display (\ref{AuxLem1Eq1}) immediately follows and thus also the statement of the lemma.

Now, after a simple relabelling (more precisely by writing $\mathsf{a}_i=\lambda(y+i+1)$) in order to establish the claim it suffices to prove the following result. Let $\{\mathsf{a}_i \}_{i\ge 0}$ be a sequence of numbers in $[s,M]$ and let $\mathsf{C}_{\mathsf{a}}^{\ge R}$ be the contour defined above. Then, uniformly for $w \in \mathsf{C}_{\mathsf{a}}^{\ge R}$ we have:
\begin{align*}
\sum_{l=0}^{\infty}\frac{1}{\mathsf{a}_l}\prod_{i=0}^{l}\frac{\mathsf{a}_i}{\mathsf{a}_i-w}=-\frac{1}{w}.
\end{align*}
Observe that, if all the $\{\mathsf{a}_i \}_{i\ge 0}$ are equal this is just a geometric series. We claim that we have the following key identity for finite $k$:
\begin{align*}
\sum_{l=0}^{k}\frac{1}{\mathsf{a}_l}\prod_{i=0}^{l}\frac{\mathsf{a}_i}{\mathsf{a}_i-w}=-\frac{1}{w}\left(1-\prod_{l=0}^k\frac{\mathsf{a}_l}{\mathsf{a}_l-w}\right).
\end{align*}
This is a consequence (by induction) of the trivial to check equality:
\begin{align*}
-\frac{1}{w}\left(1-\prod_{l=0}^k\frac{\mathsf{a}_l}{\mathsf{a}_l-w}\right)+\frac{1}{\mathsf{a}_{k+1}}\prod_{l=0}^{k+1}\frac{\mathsf{a}_l}{\mathsf{a}_l-w}=-\frac{1}{w}\left(1-\prod_{l=0}^{k+1}\frac{\mathsf{a}_l}{\mathsf{a}_l-w}\right).
\end{align*}
Since the contour $\mathsf{C}_{\mathsf{a}}^{\ge R}$ was chosen so that for any $k\ge 1$:
\begin{align*}
\sup_{w \in \mathsf{C}_{\mathsf{a}}^{\ge R}}\prod_{l=0}^k\bigg|\frac{\mathsf{a}_l}{\mathsf{a}_l-w}\bigg|\le \frac{1}{R^{k+1}}
\end{align*}
with $R>1$, the result readily follows.
\end{proof}

\begin{proof}[Proof of Lemma \ref{AuxLemma2}]
We prove this by induction on $k$. For the base case $k=1$ it suffices to observe that:
\begin{align*}
\frac{1}{2\pi \i} \oint_{\mathsf{C}_{\lambda}}\frac{\psi_y(w)p_x(w)}{w}dw=\mathbf{1}(x>y), \ x,y \in \mathbb{Z}_+.
\end{align*}
For the inductive step, we first compute:
\begin{align*}
\phi^{(k+1)}(y,x)=-\frac{1}{\lambda(y)}\sum_{z>y}^{}-\frac{1}{\lambda(z)}\frac{1}{2\pi \i}\oint_{\mathsf{C}_{\lambda}}\frac{\psi_z(w)p_x(w)}{w^k}dw.
\end{align*}
We can then deform the contour $\mathsf{C}_{\lambda}$ to $\mathsf{C}_{\lambda}^{\ge R}$ as in the proof of Lemma \ref{AuxLemma1} and use (\ref{KeyEquality}) to conclude.
\end{proof}

\begin{proof}[Proof of Lemma \ref{AuxLemma3}]
We make use of Lemma \ref{AuxLemma2} and apply the same arguments as in the proof of Lemma \ref{AuxLemma1} to compute:
\begin{align*}
\phi^{(k)}(\mathsf{virt},x)&=\sum_{y\ge 0}^{}\phi^{(k-1)}(y,x)=\sum_{y\ge 0}^{}-\frac{1}{\lambda(y)} \frac{1}{2\pi \i} \oint_{\mathsf{C_\lambda}}\psi_y(w)\frac{p_x(w)}{w^{K-1}}dw\\
&=\sum_{y\ge 0}^{}-\frac{1}{\lambda(y)} \frac{1}{2\pi \i} \oint_{\mathsf{C}^{\ge R}_\lambda}\psi_y(w)\frac{p_x(w)}{w^{k-1}}dw=\frac{1}{2\pi \i}\oint_{\mathsf{C}^{\ge R}_\lambda}\frac{p_x(w)}{w^{k}}dw=\frac{1}{2\pi \i}\oint_{\mathsf{C}_0}\frac{p_x(w)}{w^{k}}dw.
\end{align*}
\end{proof}

\begin{proof}[Proof of Lemma \ref{AuxLemma4}]
We first write using the explicit expression:
\begin{align}\label{AuxLem4Eq1}
\sum_{x\ge 0}^{}\Psi_i^n(x)\Phi_j^n(x)=\sum_{x\ge 0}\left(-\frac{1}{\lambda(x)}\right)\frac{1}{2\pi \i} \oint_{\mathsf{C}_{\lambda}}\psi_x(w)w^{i}e^{-tw}dw\frac{1}{2\pi \i} \oint_{\mathsf{C}_0}\frac{p_x(u)}{e^{-tu}u^{j+1}}du.
\end{align}
We claim that for any $l\in \mathbb{Z}_+$, uniformly for $u \in \mathfrak{K}$, where $\mathfrak{K}$ is an arbitrary compact neighbourhood of the origin, we have:
\begin{align*}
\sum_{x\ge 0}^{}p_x(u)\left(-\frac{1}{\lambda(x)}\right)\frac{1}{2\pi \i}\oint_{\mathsf{C}_{\lambda}}\psi_x(w)w^le^{-tw}dw=e^{-tu}u^l.
\end{align*}
Then, (\ref{AuxLem4Eq1}) becomes
\begin{align*}
\frac{1}{2\pi \i}\oint_{\mathsf{C}_0}e^{tu}\frac{1}{u^{j+1}}e^{-tu}u^idu=\mathbf{1}(i=j).
\end{align*}
In order to establish the claim, it is equivalent (by taking finite linear combinations, since $p_x(\cdot)$ is a polynomial of degree $x$) to prove that for any $l\in \mathbb{Z}_+$, uniformly for $u\in \mathfrak{K}$ (an arbitrary compact neighbourhood of the origin):
\begin{align*}
\sum_{x\ge 0}^{}p_x(u)\left(-\frac{1}{\lambda(x)}\right)\frac{1}{2\pi \i}\oint_{\mathsf{C}_{\lambda}}\psi_x(w)p_l(w)e^{-tw}dw&=e^{-tu}p_l(u),\\
\sum_{x\ge 0} e^{t\mathsf{L}}(l,x)p_x(u)&=e^{-tu}p_l(u).
\end{align*}
Now, observe that for fixed $u$, the function $x\mapsto p_x(u)$ is an eigenfunction, with eigenvalue $-u$, of the generator $\mathsf{L}$:
\begin{align*}
\mathsf{L}_xp_x(u)=\lambda(x)\left[p_{x+1}(u)-p_x(u)\right]=p_x(u)\lambda(x)\left[\frac{\lambda(x)-u}{\lambda(x)}-1\right]=-up_x(u), \ x \in \mathbb{Z}_+.
\end{align*}
Thus, for $u$ fixed we have:
\begin{align*}
\left[e^{t\mathsf{L}}p_{\cdot}(u)\right](l)=e^{-tu}p_l(u), \ l \in \mathbb{Z}_+.
\end{align*}
Now, in order to show that the convergence is uniform for $u \in \mathfrak{K}$ we proceed as follows. We first estimate:
\begin{align*}
\sup_{u\in \mathfrak{K}}|p_x(u)|=\sup_{u\in \mathfrak{K}} \prod_{k=0}^{x-1}\bigg|1-\frac{u}{\lambda(k)}\bigg| \le \mathsf{const}(s,\mathfrak{K})^x, \ x \in \mathbb{Z}_+.
\end{align*}
On the other hand, making use of the spectral expansion (\ref{SpectralExpansion}), for any $l\in \mathbb{Z}_+$, we have the following bound, where $R$ can be picked arbitrarily large: 
\begin{align*}
|e^{t\mathsf{L}}(l,x)|= \bigg|-\frac{1}{\lambda(x)}\frac{1}{2\pi \i} \oint_{\mathsf{C}^{\ge R}_{\lambda}}\psi_x(w)p_l(w)e^{-tw}dw\bigg|\le C \frac{1}{R^{x}}, \ x \in \mathbb{Z}_+.
\end{align*}
Here $C$ denotes a generic constant independent of $x$ (we suppress dependence of $C$ on $l$). We pick $R$ large enough so that:
\begin{align*}
\eta\overset{\textnormal{def}}{=}\frac{\mathsf{const}(s,\mathfrak{K})}{R}<1.
\end{align*}
Then using the Weirstrass M-test, since for all $x \in \mathbb{Z}_+$
\begin{align*}
\sup_{u\in \mathfrak{K}}|e^{t\mathsf{L}}(l,x)p_x(u)|\le C \eta^x, \ \textnormal{where} \ \eta <1,
\end{align*}
we get that, for any $l \in \mathbb{Z}_+$:
\begin{align*}
\sum_{x\ge 0}^{}e^{t\mathsf{L}}(l,x)p_x(u)=e^{-tu}p_l(u), \ \textnormal{ uniformly for } u \textnormal{ on compacts } \mathfrak{K},
\end{align*}
as required.
\end{proof}

\begin{proof}[Proof of Lemma \ref{AuxLemma5}]
Let $1\le n \le N$. Using the Cauchy integral formula we see that, for $1\le k \le n$:
\begin{align*}
\phi^{(k)}(\mathsf{virt},x)=\frac{1}{(k-1)!}\frac{d^{k-1}}{dw^{k-1}}p_x(w)\bigg|_{w=0}.
\end{align*}
On the other hand, again using the Cauchy integral formula, we have for $j=0,\dots,n-1$:
\begin{align*}
\Phi_j^n(x)=\frac{1}{j!}\frac{d^{j}}{dw^{j}}\left(e^{tw}p_x(w)\right)\bigg|_{w=0}=\sum_{i=0}^{j}c_i^j\frac{d^{i}}{dw^{i}}p_x(w)\bigg|_{w=0},
\end{align*}
with $c_j^j\neq 0$. Thus, it is immediate that
\begin{align*}
span\bigg\{\Phi_0^n(\cdot),\Phi_1^n(\cdot),\dots,\Phi_{n-1}^n(\cdot)\bigg\}=V_n,
\end{align*}
as desired.
\end{proof}

\begin{rmk}\label{LevelInhomogeneousRemark5}
The computation of the correlation kernel in the level inhomogeneous setting is more complicated and notationally cumbersome and we do not pursue the details here. We simply record the key ingredient for computing the iterated convolutions as contour integrals. This is the analogue of (and in fact follows from) display (\ref{KeyEquality}): 
\begin{align*}
\sum_{z>y}^{}\frac{1}{\lambda(z)+\alpha_{n+1}}\mathsf{h}_{\alpha_n}^{\alpha_{n+1}}(z)\psi_z\left(w;\lambda(\cdot)+\alpha_{n+1}\right)&=\sum_{z>y}^{}\frac{1}{\lambda(z)+\alpha_{n}}\prod_{l=0}^{z}\frac{\lambda(l)+\alpha_n}{\lambda(l)+\alpha_n-\left(w+\alpha_n-\alpha_{n+1}\right)}\\
&=-\frac{1}{w+\alpha_n-\alpha_{n+1}}\psi_y(w+\alpha_n-\alpha_{n+1};\lambda(\cdot)+\alpha_n),
\end{align*}
holding uniformly for $w$ on some large contour $\mathsf{C}_{\lambda,\alpha_n,\alpha_{n+1}}^{\ge R}$.
\end{rmk}

\bigskip
\noindent
{\sc Mathematical Institute, University of Oxford, Oxford, OX2 6GG, UK.}\newline
\href{mailto:theo.assiotis@maths.ox.ac.uk}{\small theo.assiotis@maths.ox.ac.uk}

\end{document}